\documentclass[hidelinks,onefignum,onetabnum]{siamart220329}

\usepackage{amsfonts}
\usepackage{graphicx}
\usepackage[utf8]{inputenc}
\usepackage{amsmath}
\usepackage[english]{babel}
\usepackage{amssymb}
\usepackage{oldgerm}
\usepackage{mathrsfs}
\usepackage{comment}
\usepackage{setspace}
\usepackage{tikz}
\usepackage{pgfplots}
\usepackage[scanall]{psfrag}
\usepackage{graphicx,pst-all,bm}
\usepackage{flushend}
\usepackage{subfig}
\usepackage{mdwlist}
\usepackage{multirow}
\usepackage{paralist}
\usepackage[paper=a4paper,left=25mm,right=25mm,top=25mm,bottom=25mm]{geometry}
\usepackage{hyperref}
\usepackage{color}
\usepackage{longtable}
\usepackage{cuted}
\usepackage{verbatim}
\usepackage{siunitx}
\usepackage{fancyhdr}
\usepackage{nomencl}
\usepackage{cite}
\usepackage{float}
\usepackage{amsopn}
\usepackage{xcolor}

\usepackage{tikz}
\usetikzlibrary{shapes,arrows,decorations.markings,decorations.text}
\usetikzlibrary{arrows.meta}
\usetikzlibrary{calc}
\tikzstyle{block2} = [rectangle, draw, 
text width=25em, text centered, rounded corners, minimum height=6em]
\tikzstyle{block3} = [rectangle, draw, 
text width=30em, text centered, rounded corners, minimum height=4em]
\tikzstyle{block3b} = [rectangle, draw, 
text width=30em, text centered, rounded corners, minimum height=4em]
\tikzstyle{block4} = [rectangle, draw, text width = 18em, text centered, rounded corners, minimum height = 4em]
\tikzstyle{block5} = [rectangle, draw, text width = 18em, text centered, rounded corners, minimum height = 4em]
\tikzstyle{block7} = [rectangle, draw, 
text width=35em, text centered, rounded corners, minimum height=4em]

\tikzstyle{blocktest} = [rectangle, draw, 
text width=30em, text centered, rounded corners, minimum height=4.3em]

\tikzstyle{arrow} = [draw, -latex']
\tikzstyle{line} = [draw, -latex']

\newtheorem{Prop}{Proposition}[section]
\newtheorem{Lem}[Prop]{Lemma}
\newtheorem{Thm}[Prop]{Theorem}
\newtheorem{Cor}[Prop]{Corollary}

\newtheorem{Def}[Prop]{Definition}
\newtheorem{Rem}[Prop]{Remark}

\newtheorem{Ex}[Prop]{Example}

\newcommand{\Q}{S}
\newcommand{\PH}{\Sigma_{n,m}}

\newcommand{\rk}{\mathrm{rk}\,}
\newcommand{\im}{\mathrm{im}\,}

\newcommand{\norm}[1]{\left\Vert #1 \right\Vert}
\newcommand{\abs}[1]{\left\vert #1 \right\vert}
\newcommand{\set}[1]{\left\lbrace #1 \right\rbrace}

\renewcommand{\exp}[1]{e^{#1}}

\renewcommand{\Re}{\mathrm{Re}\,}
\renewcommand{\Im}{\mathrm{Im}\,}

\newsiamremark{remark}{Remark}
\newsiamremark{hypothesis}{Hypothesis}
\crefname{hypothesis}{Hypothesis}{Hypotheses}
\newsiamthm{claim}{Claim}
\newsiamremark{example}{Example}

\headers{Hidden regularity in optimal control of pH systems
}{T. Faulwasser, J. Kirchhoff, V. Mehrmann, F. Philipp, M. Schaller, K. Worthmann}

\title{Hidden regularity in singular optimal control of port-Hamiltonian systems
\thanks{Preprint submitted to Arxiv on 11/20/23}
\funding{JK thanks Technische Universität Ilmenau and Freistaat Thüringen for their financial support as part of the Thüringer Graduiertenförderung. VM was supported by German Research Foundation through Priority Program SPP 1984 \textit{Hybrid and multimodal energy systems} within the project \textit{Distributed Dynamic Security Control}. FP was funded by the Carl Zeiss Foundation within the project \textit{DeepTurb--Deep Learning in and from Turbulence} and by the free state of Thuringia within the project \textit{THInKI--Th\"uringer Hochschulinitiative für KI im Studium}. KW gratefully acknowledges funding by the German Research Foundation (DFG; grant no.\ 507037103).}}
\author{
Timm Faulwasser\thanks{Institute of Energy Systems, Energy
Efficiency and Energy Economics, TU Dortmund University,  Germany (\email{timm.faulwasser@ieee.org})}
\and Jonas Kirchhoff\thanks{Optimization-based Control Group, Institute of Mathematics, Technische Universität Ilmenau, Germany
(\{jonas.kirchhoff, friedrich.philipp, manuel.schaller, karl.worthmann\}@tu-ilmenau.de).
}
\and Volker Mehrmann\thanks{Institute of Mathematics, Technische Universität Berlin, Germany (\email{mehrmann@math.tu-berlin.de}), }
\and Friedrich Philipp\footnotemark[3]
\and Manuel Schaller\footnotemark[3]
\and Karl Worthmann\footnotemark[3]}

\newcommand{\R}{\ensuremath{\mathbb R}}    
\newcommand{\C}{\ensuremath{\mathbb C}}    
\newcommand{\K}{\ensuremath{\mathbb K}}    

\newcommand{\calA}{\mathcal A}

\newcommand{\calE}{\mathcal E}

\newcommand{\la}{\lambda}

\newcommand{\bmat}[4]
{
   \begin{bmatrix}
      #1 & #2\\
      #3 & #4
   \end{bmatrix}
}
\newcommand{\bvec}[2]
{
   \begin{bmatrix}
      #1\\
      #2
   \end{bmatrix}
}
\newcommand{\sbvec}[2]{\left[\begin{smallmatrix}#1\\#2\end{smallmatrix}\right]}

\renewcommand{\Im}{\operatorname{Im}}
\renewcommand{\Re}{\operatorname{Re}}

\newcommand{\ol}{\overline}
\newcommand{\ul}{\underline}

\DeclareMathOperator*{\argmin}{argmin}

\begin{document}
\maketitle
\begin{abstract}
We study the problem of state transition on a finite time interval with minimal energy supply for linear port-Hamiltonian systems. While the cost functional of minimal energy supply is intrinsic to the port-Hamiltonian structure, the necessary conditions of optimality resulting from Pontryagin's maximum principle may yield singular arcs. The underlying reason is the linear dependence on the control, which makes the problem of determining the optimal control as a function of the state and the adjoint more complicated or even impossible. To resolve this issue, we fully characterize regularity of the (differential-algebraic) optimality system by using the interplay of the cost functional and the dynamics. In case of the optimality DAE being characterized by a regular matrix pencil, we fully determine the control on the singular arc. In case of singular matrix pencils of the optimality system, we propose an approach to compute rank-minimal quadratic perturbations of the objective such that the optimal control problem becomes regular. We illustrate the applicability of our results by a general second-order mechanical system and a discretized boundary-controlled heat equation.
\end{abstract}







\section{Introduction}
Since their introduction  in~\cite{Maschke93} and especially in recent years, port-Hamiltonian systems have been successfully applied in the modeling of a wide variety of physical processes: mechanics \cite{brugnoli2019por1t,brugnoli2019port2,siuka2011port,macchelli2007port},  electrical engineering \cite{reis2021analysis,schoberl2008modelling}, thermodynamics and fluid dynamics \cite{Altmann2017,rashad2021port,rashad2021port2,hoang2011port}, economics \cite{macchelli2014towards}, see also \cite{SchaJelt2014,duindam2009modeling,Jacob2012} and the recent survey \cite{MehU23}.

Recently, also optimal control of port-Hamiltonian systems has been given some attention, see, e.g., \cite{kolsch2021optimal,wu2020reduced}. Since port-Hamiltonian modeling is energy-based, the energy supplied to the system is readily available as a mathematical expression in terms of the ports and it is natural from an application point of view to employ it as the objective minimized in optimal control problems (OCPs), cf.~\cite{FaulMascPhilSchaWort21,FaulMascPhilSchaWort22}. For an application to energy-optimal building control we refer to \cite{SchaZell23} and entropy-optimal control is considered in~\cite{lhmnc,PhilScha23}. 
However, the resulting OCP is a singular and standard solution techniques, e.g.\ construction of Riccati state feedback, are not directly applicable, see e.g. \cite{Meh91,KunM11a}. One could instead employ general techniques for singular optimal control like those studied in
\cite{Aro11,BelJ75}. But these approaches would not exploit 
the particular structure of the port-Hamiltonian system and the chosen cost function. Instead we apply and extend techniques developed in~\cite{Camp76} to derive, under weaker conditions than regularity of the optimal control problem, a condition that guarantees the existence of an optimal solution with a closed form for the associated optimal control. This weaker condition is related to the regularity of the matrix pencil, which is associated with the necessary optimality condition arising from a Pontryagin maximum principle. 
The derivation of this general existence condition using the port-Hamiltonian structure of the OCP is at the center of our work. In the case where this existence condition fails to hold, we  provide a regularization procedure via minimal-rank perturbations of the cost functional such that the pencil associated with the optimality condition becomes regular.

For a fixed time horizon $T>0$, the problem of transferring an initial state $x_S\in\K^n$ of a port-Hamiltonian system under control constraints to a desired state $x_T\in\K^n$ with minimal energy supply can be formulated as
\begin{align}\label{eq:special_port-Hamiltonian_OCP0T}
\begin{split}
\min_{u\in L^1([0,T],\mathbb{K}^m)}\int_{0}^{T} \Re &\left(y(t)^H u(t)\right)\,\mathrm{d}t\\
\mathrm{s.t.}\quad  \tfrac{\mathrm{d}}{\mathrm{d}t}x(t) &= (J-R)Qx(t)+Bu(t) \quad \mathrm{for a.e.}\ t\in [0,T],\\ x(0) &= x_S,\ x(T) = x_T,\\
y(t) &= B^HQ x(t), \quad u(t) \in \mathbb{U} \quad \mathrm{for a.e.}\ t\in [0,T],
\end{split}
\end{align}
where $\mathbb{U} \subset \mathbb{K}^m$, $m\le n$, is a non-empty, convex and compact set, and
\begin{align*}
(J,R,Q,B)\in\PH := \set{(J,R,Q,B)\in\left(\mathbb{K}^{n\times n}\right)^3\times\mathbb{K}^{n\times m}\,\Big\vert\, J = -J^H,\, R = R^H\geq 0,\, Q = Q^H\geq 0}.
\end{align*}
The OCP \eqref{eq:special_port-Hamiltonian_OCP0T} was studied in~\cite{FaulMascPhilSchaWort21,FaulMascPhilSchaWort22}, see also \cite{lhmnc} for a particular infinite-dimensional case. While in these papers it has been shown that, under a reachability condition of the terminal state, an optimal solution of \eqref{eq:special_port-Hamiltonian_OCP0T} typically exists, one can, a priori, not expect its uniqueness nor that the associated optimal control can be expressed as a state or output feedback in terms of the optimal trajectory and the associated Lagrange multiplier (also called costate or adjoint) in a unique way. To illustrate this issue, we consider the particular case of box constraints, that is, 
\begin{align*}
\mathbb{U} = \set{z\in\K^m\,\big\vert\, \Re z_j\in[\ul v_j,\ol v_j],\;\Im z_j\in[\ul w_j,\ol w_j]\text{ for all $j\in\{1,\ldots,m\}$}},
\end{align*}
where $\ul v_j,\ul w_j,\ol v_j,\ol w_j\in\R$ with $\underline{v}_j\leq \overline{v}_j$ and $\underline{w}_j\leq\overline{w}_j$ for all $j\in\set{1,\ldots,m}$. By the Pontryagin maximum principle, see e.g.~\cite{Libe12,PonBGM62,Lee67}, for an optimal state-control pair $(x,u)\in W^{1,1}([0,T],\mathbb{K}^n)\times L^1([0,T];\mathbb{K}^m)$ there exist $\lambda \in W^{1,1}([0,T],\mathbb{K}^n)$ and a constant $\lambda_0\in\mathbb{K}$ such that $(\lambda(t),\lambda_0)\neq 0$ for all $t\in[0,T]$ and
\begin{subequations}\label{eq:intro_optcond}
\begin{align}
\tfrac{\mathrm{d}}{\mathrm{d}t}{\lambda}(t) & = -\big((J-R)Q\big)^H\lambda(t) - \lambda_0 QBu(t),\\
\tfrac{\mathrm{d}}{\mathrm{d}t}{x}(t) & = (J-R)Q x(t) + Bu(t),\\
u(t) & \in \argmin_{\bar{u}\in \mathbb{U}} \,\Re\left(\bar{u}^H B^H\left(\lambda_0 Qx(t) + \lambda(t)\right)\right) \label{eq:intro_optcond_control}
\end{align}
\end{subequations}
for a.e. $t\in[0,T]$. 
Using the \emph{switching function} $s(t) = B^H(\lambda_0 Qx(t) + \lambda(t))$, the pointwise minimization in \eqref{eq:intro_optcond_control} yields the following componentwise characterization for almost all $t\in[0,T]$:
\begin{align}\label{eq:charac}
\begin{split}
&\Re s_j(t) > 0 \quad \mbox{implies} \quad \Re u_j(t) = \underline{v}_j \qquad \text{and} \qquad
\Im s_j(t) > 0 \quad \mbox{implies} \quad \Im u_j(t) = \ol{w}_j,\\
&\Re s_j(t) < 0 \quad\mbox{implies}\quad \Re u_j(t) = \overline{v}_j \qquad \text{and} \qquad 
\Im s_j(t) < 0 \quad\mbox{implies}\quad \Im u_j(t) = \ul{w}_j.
\end{split}
\end{align}
However, if {$\Re s_j(\cdot)$ vanishes on a set with non-zero Lebesgue measure, 
then \eqref{eq:intro_optcond_control} does not provide any information on $\Re u_j(\cdot)$ on that set. The same holds analogously for the imaginary part. If the set is an interval, the interval is called a {\em singular arc}, cf.\ \cite{Aro11,zelikin2012theory}.

In \cite[Theorem 8]{FaulMascPhilSchaWort21} it was shown that the simple condition $\im B\cap\ker RQ = \{0\}$ allows for the unique characterization of the optimal control on singular arcs in terms of the optimal state $x$ and the adjoint~$\lambda$. In Proposition \ref{prop:indexthree} below, we prove that the condition $\im B\cap\ker RQ = \{0\}$ can be characterized in terms of the Jordan structure associated with the eigenvalue $\infty$ (the \emph{Kronecker index}) of the differential-algebraic optimality system along singular arcs.

In this paper, we are mainly interested in characterizations 
of the optimal control on singular arcs, i.e.\ intervals on which some components of the switching function vanish. 
To illustrate our approach, let us assume the control constraints are inactive on an interval $[t_0,t_1]$, $0\leq t_0< t_1 \leq T$. Consequently, all components of the switching function vanish on $[t_0,t_1]$ due to the contraposition of \eqref{eq:charac}. Then, the pointwise inclusion \eqref{eq:intro_optcond_control} may be replaced with the algebraic condition $s(t) = 0$ for all $t\in [t_0,t_1]$ such that~\eqref{eq:intro_optcond} is a differential-algebraic equation (DAE).
Conversely, this DAE is the optimality system of the unconstrained counterpart of \eqref{eq:special_port-Hamiltonian_OCP0T}, i.e.,
\begin{align}\label{eq:special_port-Hamiltonian_OCP}
\begin{split}
\min_{u\in L^1([t_0,t_1],\mathbb{K}^m)}\int_{t_0}^{t_1} \Re &\left(y(t)^H u(t)\right)\,\mathrm{d}t\\
\mathrm{s.t.}\quad
\tfrac{\mathrm{d}}{\mathrm{d}t}x(t) &= (J-R)Qx(t)+Bu(t)\quad \mbox{for a.e.}\ t\in [t_0,t_1],\\ x(t_0) &= x^0,\ x(t_1) = x^1,\\
y(t) &= B^HQ x(t)\quad \mbox{for a.e.}\ t\in [t_0,t_1],
\end{split}
\end{align}
where $x^0\in\K^n$ and $x^1\in\K^n$ are the states at which the optimal solution of \eqref{eq:special_port-Hamiltonian_OCP0T} enters, respectively, \ leaves the singular arc.

Typically, 
only some (and not all) of the control constraints are inactive. Then, denoting the active and inactive set corresponding to the control constraints by $\mathcal{A}\subset \{1,\ldots,m\}$ and 
$\mathcal{I} = \{1,\ldots,m\}\setminus \mathcal{A}$, respectively, 
after suitable permutation one can split the control into $u = \begin{bmatrix}u_{\mathcal{I}}^\top,\,u_{\mathcal{A}}^\top\end{bmatrix}^\top$, which yields the dynamics 
\begin{align*}
    \frac{\mathrm{d}}{\mathrm{dt}} x(t) = (J-R)Qx(t) + B_\mathcal{I}u_\mathcal{I}(t) + B_\mathcal{A}u_{\mathcal{A}}(t).
\end{align*}
Here, $u_{\mathcal{A}}(t)$ is determined by the pointwise 
minimization~\eqref{eq:intro_optcond} and can, thus, be treated as in inhomogeneity
$f(t) = B_\mathcal{A}u_{\mathcal{A}}(t)$ in the state equation of~\eqref{eq:special_port-Hamiltonian_OCP}. This splitting has to be adapted whenever the \emph{active set}~$\mathcal{A}$ changes. 
However, for simplicity of exposition, we will analyse to the homogeneous case of \eqref{eq:special_port-Hamiltonian_OCP}}. All our considerations are generalizable to characterize the inactive controls in case of active constraints as the inhomogenities only enter the differential equations in \eqref{eq:intro_optcond}.



The approach in \cite{Camp76} studies the optimality conditions associated with~\eqref{eq:special_port-Hamiltonian_OCP} obtained by the Pontryagin maximum principle. When this optimality system is regular in the DAE sense, then the optimal control is linearly determined by the optimal state trajectory and associated adjoint. Moreover, criteria on the existence and uniqueness of an optimal control can be deduced, see also \cite{Meh91} for a detailed discussion and numerical methods.

For cases where this regularity is not given, one can resort to regularizations of the OCP 
via ``small'' perturbations of the original cost functional by a ``small'' quadratic term in the control and obtains the regularized optimal control problem 
\begin{equation}\label{eq:perturbed_port-Hamiltonian_OCP}
\begin{aligned}
\min_{u\in L^1([t_0,t_1],\mathbb{K}^m)}\int_{t_0}^{t_1} \big[\Re&(y(t)^H u(t)) + u(t)^H\Q u(t)\big]\,\mathrm{d}t\\
\mathrm{s.t.}\quad  \tfrac{\mathrm{d}}{\mathrm{d}t}x(t) &= (J-R)Qx(t)+Bu(t)\quad \mbox{for a.e.\ }\ t\in [t_0,t_1],\\x(t_0) &= x^0,\ x(t_1) = x^1,\\
 y(t) &= B^HQ x(t)\quad \mbox{for a.e.\ }\ t\in [t_0,t_1]
\end{aligned}
\end{equation}
for some $\Q\in\mathbb{K}^{m\times m}$ with $\Q^H = \Q\geq 0$. In order to stay ``close'' to the original problem \eqref{eq:special_port-Hamiltonian_OCP}, we aim to choose $\Q$ ``minimally'' (in the sense that the rank of $\Q$ is minimal) so that the associated optimality DAE is regular. 
An alternative type of smallness of the regularization would be to set $S = \alpha I$ with small $\alpha > 0$, which corresponds to a Tikhonov regularization 
as is typical in optimal control of partial differential equations \cite{Tro10}.

This work is structured as follows. First, we recall the necessary and sufficient optimality conditions obtained from the Pontryagin maximum principle. These conditions are given as a linear DAE in the state, the adjoint, and the control. Solution formulas for this DAE in terms of the Drazin inverse can be applied if the associated pencil is regular. Therefore, we study the properties of this pencil in Section~\ref{sec:Optimality_pencil}. Our results are summarized 
in Figure~\ref{fig:Schaubild_2}. 
In particular, in Subsection~\ref{subsec:reg}, we provide characterisations of regularity which extend Campbell's results~\cite{Camp76} and relate them to the sufficient regularity condition 
$\ker R\cap\im B = \set{0}$ proposed in \cite[Theorem 3.8]{FaulMascPhilSchaWort21}. Then, in Subsection~\ref{subsec:minreg}, we use the characterisations of regularity to find a rank-minimal regularization term $\Q$. In Subsection~\ref{subsec:index}, it is shown that the condition $\ker R\cap\im B = \set{0}$ characterises the case that the index of the optimality pencil is three. 
In Section~\ref{sec:existence}, we study existence and uniqueness of an optimal solution under the assumption of a regular optimality pencil. In particular, a state-feedback law is deduced. Lastly, we illustrate the results with some examples.


\newpage
{\bf Nomenclature}

\smallskip\noindent
\hspace*{-3mm}
\begin{tabular}[ht]{rl}
$\K$ & Either $\mathbb{R}$ or $\mathbb{C}$\\
$\K^{n\times m}$
  &   $n\times m$ matrices with entries in $\K$\\
$A^H$ & Conjugate transpose of a matrix $A\in K^{n\times m}$\\
$\sigma(A)$ & Eigenvalues of $A\in\K^{n\times n}$\\
$A\geq 0$ & Means $x^H Ax\geq 0$ for all $x\in\K^n$\\
$A>0$ & Means $x^H Ax>0$ for all $x\in\K^n\setminus\set{0}$\\
$\Sigma_{n,m}$ & $\left\{(J,R,Q,B)\in\left(\mathbb{K}^{n\times n}\right)^3\times\mathbb{K}^{n\times m}\,\Big\vert\, J = -J^H,\, R = R^H\geq 0,\, Q = Q^H\geq 0\right\}$\\
$L^1([0,T],\K^n)$ & Lebesgue measureable and integrable functions $f:[0,T]\to\K^n$\\
$\left.f\right|_{W}$ & Restriction of the function $f:V\!\to\!\R^n$ to $W\subseteq V$\\
$\mathcal{C}^k(I,\K^n)$ & $k$-times continuously differentiable functions $f:I\to\K^n$\\ 
$\mathcal{C}^k([0,T],\K^n)$ & $:=\set{f\vert_{[0,T]}\,\big\vert\,\exists\,\varepsilon>0: f\in\mathcal{C}^1(-\varepsilon,T+\varepsilon),\K^n}$\\
$\mathcal{C}^k([0,T],\K^n)$ & $:=\set{f\vert_{[0,T]}\,\big\vert\,\exists\,\varepsilon>0: f\in\mathcal{C}^1(-\varepsilon,T+\varepsilon),\K^n}$\\
$W^{1,1}([0,T],\K^n)$ & Absolutely continuous functions $f:[0,T]\to\K^n$ with derivative $f'\in L^1([0,T],\K^n)$\\
$V[t]$ & Polynomials with entries in the vector space $V$\\
$V[[t]]$ & Formal power series with entries in the vector space $V$
\end{tabular}

\section{Necessary optimality condition}

Before we derive necessary optimality conditions, we transform the objective into a quadratic cost functional without mixed terms. Since the dynamics of \eqref{eq:perturbed_port-Hamiltonian_OCP} have port-Hamiltonian structure with quadratic Hamiltonian energy function $x\mapsto \frac{1}{2} x^HQx$, we have the well-known \emph{energy balance equation}, see e.g. \cite{SchaJelt2014},
\begin{equation}\label{eq:energy_balance}
\begin{aligned}
\frac{1}{2}x(t_1)^HQ x(t_1) - \frac{1}{2}x(t_0)^HQ x(t_0) & = \int_{t_0}^{t_1}\big[\Re y(t)^Hu(t) - x(t)^HQRQx(t)\big]\,\mathrm{d}t
\end{aligned}
\end{equation}
for all tuples $(x,y,u)\in W^{1,1}([t_0,t_1],\mathbb{K}^n)\times W^{1,1}([t_0,t_1],\mathbb{K}^n)\times L^1([t_0,t_1],\mathbb{K}^m)$ satisfying the system
\begin{equation}\label{eq:pH_ISO-system}
\begin{aligned}
\tfrac{\mathrm{d}}{\mathrm{d}t}x(t) & = (J-R)Qx(t)+Bu(t),\\
y(t) & = B^HQ x(t)
\end{aligned}
\end{equation}
for a.e.\ $t\in[t_0,t_1]$. Thus, again for tuples $(x,y,u)$ satisfying \eqref{eq:pH_ISO-system}, the cost functional of \eqref{eq:perturbed_port-Hamiltonian_OCP} can be rewritten as
\begin{align*}
\int_{t_0}^{t_1} \big[\Re y(t)^H u(t) + u(t)^H\Q u(t)\big]\,\mathrm{d}t = &\ \frac{1}{2}x(t_1)^HQ x(t_1) - \frac{1}{2}x(t_0)^HQx(t_0)\\&\quad + \int_{t_0}^{t_1}\big[x(t)^HQRQx(t) + u(t)^H\Q u(t)\big]\,\mathrm{d}t.
\end{align*}
The OCP \eqref{eq:perturbed_port-Hamiltonian_OCP} includes the initial condition $x(t_0)=x^0$ and the terminal constraint $x(t_1)=x^1$. 
Thus, $\frac{1}{2}x(t_1)^HQ x(t_1) - \frac{1}{2}x(t_0)^HQ x(t_0)$ is a constant on the set of feasible trajectories of~\eqref{eq:perturbed_port-Hamiltonian_OCP} and can be neglected in the minimization. Therefore, we may equivalently consider the linear-quadratic optimal control problem
\begin{align}\label{eq:ultimate_port-Hamiltonian_OCP}
\begin{split}
\min_{u\in L^1([t_0,t_1],\mathbb{K}^m)}\frac{1}{2}\int_{t_0}^{t_1} \big[&x(t)^HQRQx(t)  + u(t)^H\Q u(t)\big]\,\mathrm{d}t\\
\mathrm{s.t.}\quad  \tfrac{\mathrm{d}}{\mathrm{d}t}x(t) &= (J-R)Qx(t)+Bu(t)\quad \mathrm{for a.e.\ } t\in [t_0,t_1],\\
x(t_0) &= x^0,\ x(t_1) = x^1,
\end{split}
\end{align}
which has the same optimal solutions 
as~\eqref{eq:perturbed_port-Hamiltonian_OCP}. 
The stage cost 
of~\eqref{eq:ultimate_port-Hamiltonian_OCP} is the quadratic function
\begin{align*}
    g: \mathbb{K}^n\times\mathbb{K}^{m}\to\mathbb{K},\qquad (x,u)\mapsto \frac{1}{2} \begin{bmatrix}
x\\u
\end{bmatrix}^H\begin{bmatrix}
QRQ &0\\
0 & \Q
\end{bmatrix}\begin{bmatrix}
x\\u
\end{bmatrix}.
\end{align*}
Then, the associated optimal control Hamiltonian function 
reads
\begin{align*}
H:\mathbb{K}^n\times\mathbb{K}^m\times\mathbb{K}^n\times \mathbb{K}\to\mathbb{K},\quad (x,u,\lambda,\lambda_0) \mapsto \lambda^H(J-R)Qx + \lambda^HBu + \lambda_0 g(x,u).
\end{align*}
By the Pontryagin maximum principle (see e.g.~\cite[p.102]{Libe12}), one has the following necessary optimality condition: If $(x,u)\in W^{1,1}([t_0,t_1],\mathbb{K}^n)\times L^1([t_0,t_1],\mathbb{K}^m)$ is an optimal trajectory of~\eqref{eq:ultimate_port-Hamiltonian_OCP}, then there exists a function $\lambda\in W^{1,1}([t_0,t_1],\mathbb{K}^n)$ and a constant $\lambda_0\in\mathbb{K}$ so that $(\lambda(t),\lambda_0)\neq (0,0)$ for all $t\in [t_0,t_1]$ and the DAE
\begin{subequations}\label{eq:the_optimality_condition}
\begin{align}
-\tfrac{\mathrm{d}}{\mathrm{d}t}{\lambda} & = \big((J-R)Q\big)^H\lambda + \lambda_0 QRQx,\label{eq:zeile_1}\\
\tfrac{\mathrm{d}}{\mathrm{d}t}{x} & = (J-R)Q x + Bu,\label{eq:zeile_2}\\
0 & = -B^H\lambda - \lambda_0 \Q u\label{eq:zeile_3}
\end{align}
\end{subequations}
with boundary conditions $x(t_0) = x^0$ and $x(t_1) = x^1$ is fulfilled for almost every $t\in[t_0,t_1]$. 
We briefly state a result on normalization of the stage-cost adjoint 
$\lambda_0$ under a controllability assumption.
\begin{Lem}
    If $((J-R)Q,B)$ is controllable, i.e. $\rk\begin{bmatrix} B,(J-R)QB,\ldots,((J-R)Q)^{n-1}B\end{bmatrix} = n$,
 then the scalar multiplier in \eqref{eq:the_optimality_condition} satisfies $\la_0\neq 0$.
\end{Lem}
\begin{proof}
We abbreviate $A = (J-R)Q$ and assume that $\la_0=0$. Then,~\eqref{eq:zeile_3} yields $B^H\la=0$ and~\eqref{eq:zeile_1} reads as $\tfrac{\mathrm{d}}{\mathrm{d}t}\la = -A^H\la$, hence $B^HA^H\la(t) = -B^H\tfrac{\mathrm{d}}{\mathrm{d}t}\la(t) = 0$ for almost every $t\in[t_0,t_1]$. From the continuity of $\lambda$, we conclude that $B^HA^H\lambda\equiv 0$ on $[t_0,t_1]$. Inductively, one obtains $B^H(A^k)^H\la\equiv 0$ on $[t_0,t_1]$ for $k=0,1,\ldots,n-1$. By the controllability of $(A,B)$, this implies that $(\la(t),\la_0) = (0,0)$ for all $t\in [t_0,t_1]$. But this contradicts the non-triviality condition $(\lambda(t),\lambda_0)\neq (0,0)$ for all $t\in [t_0,t_1]$ of the Pontryagin maximum principle.
\end{proof}
In view of the previous 
lemma, in all theoretical discussions below we assume that $\la_0\neq 0$. In this case, dividing~\eqref{eq:zeile_1} and~\eqref{eq:zeile_3} by $\lambda_0$ and rescaling~$\lambda$, 
we can without loss of generality normalize $\lambda_0 = 1$.

An important result, establishing a one-to-one relationship of the optimality system \eqref{eq:the_optimality_condition} with $\lambda_0\neq 0$ and the optimal control problem~\eqref{eq:ultimate_port-Hamiltonian_OCP} is the following theorem ensuring also sufficiency of the necessary optimality conditions.

\begin{Thm}[\hspace*{-.4em}{\cite[Theorem 1]{Camp76}}]
Let $(\lambda,x,u)\in \left(W^{1,1}([t_0,t_1],\mathbb{K}^n)\right)^2\times L^2([t_0,t_1],\mathbb{K}^m)$ be a solution of~\eqref{eq:the_optimality_condition} with $\lambda_0=1$ satisfying $x(t_0) = x^0$ and $x(t_1) = x^1$. Then $(x,u)$ is a solution of the optimal control problem~\eqref{eq:ultimate_port-Hamiltonian_OCP}.
\end{Thm}

Taking additionally the structural properties of $(J,R,Q,B)\in\PH$ into account, the \emph{optimality system} \eqref{eq:the_optimality_condition} is equivalent to the following \emph{self-adjoint DAE}, see \cite{KunM23,KunMS14}:
\begin{align}\label{eq:the_optimality_DAE}
\frac{\mathrm{d}}{\mathrm{d}t} \underbrace{\begin{bmatrix}
0_{n\times n} & I_n & 0_{n\times m}\\
-I_n & 0_{n\times n} & 0_{n\times m}\\
0_{m\times n} & 0_{m\times n} & 0_{m\times m}
\end{bmatrix}}_{=:\mathcal{E}}\begin{bmatrix}
\lambda\\
x\\
u
\end{bmatrix} = \underbrace{\begin{bmatrix}
0_{n\times n} & (J-R)Q & B\\
((J-R)Q)^H & QRQ & 0_{n\times m}\\
B^H & 0_{m\times n} & \Q
\end{bmatrix}}_{=:\mathcal{A}_\Q}\begin{bmatrix}
\lambda\\x\\u
\end{bmatrix}.
\end{align}

Since we are in the constant coefficient case, we can analyse the properties of \eqref{eq:the_optimality_condition} via the properties of the \emph{optimality pencil} $s\mathcal E - \mathcal A_S$, which actually has a special symmetry structure that can be effectively used in numerical methods for solving the boundary value problem,  see \cite{ByeMX07}. Before we continue with the analysis of the pencil, we first consider an example of a second-order mechanical system.

\begin{Ex}\label{ex:mechsys}
A large class of systems from elastomechanics (possibly after a suitable finite-element discretization) can be formulated as a second-order control system
\begin{align}\label{mechsys}
M\tfrac{\mathrm{d}^2}{\mathrm{d}t^2}{q} + D\tfrac{\mathrm{d}}{\mathrm{d}t}{q} + Kq = u
\end{align}
with symmetric matrices $M,D,K\in\mathbb{R}^{\ell \times \ell}$, $\ell\in \mathbb{N}$, where the mass matrix $M$ and the stiffness matrix $K$ are positive definite, respectively, and the damping matrix $D$ is positive semi-definite.In the following we consider the fully damped case such that $D$ is positive definite. Performing a first -order formulation by composing the state $x=\sbvec{x_1}{x_2}\in\mathbb{R}^{2\ell}$ from momentum $x_1 = M\dot q$ and position $x_2 = q$, the system \eqref{mechsys} can equivalently be written in the form
\begin{align*}
\tfrac{\mathrm{d}}{\mathrm{d}t}{x} = 
\Bigg(
\underbrace{
\begin{bmatrix}
0_{\ell\times \ell} & -I_\ell\\
I_\ell & 0_{\ell\times \ell}
\end{bmatrix}}_{J}
-
\underbrace{
\begin{bmatrix}
D & 0_{\ell\times \ell}\\
0_{\ell\times\ell} & 0_{\ell\times\ell}
\end{bmatrix}}_{R}
\Bigg)
\underbrace{
\begin{bmatrix}
M^{-1} & 0_{\ell\times\ell}\\
0_{\ell\times \ell} & K
\end{bmatrix}}_{Q}x
+
\underbrace{
\begin{bmatrix}
I_\ell\\ 0_{\ell\times \ell}
\end{bmatrix}}_{B}u,
\end{align*}
see for example~\cite{WarsBohmSawoTari21} for an application in the control of high-rise buildings. We observe that $\ker RQ\cap\im B = \{0\}$ holds; that is, the sufficient condition for regularity derived in~\cite{FaulMascPhilSchaWort21}  is satisfied. Thus, we can inspect the corresponding optimal control problem \eqref{eq:perturbed_port-Hamiltonian_OCP} without regularization term, i.e. $S=0$. Further, using the \emph{Hautus Lemma}, see e.g.~\cite{Son13}, it is immediate that the pair $((J-R)Q,B)$ is controllable. Hence, we may assume that $\la_0=1$ in \eqref{eq:the_optimality_condition}.

Decomposing 
also $\la = \sbvec{\la_1}{\la_2}\in\mathbb{R}^{2\ell}$ 
analogously to the state,~\eqref{eq:zeile_3} is equivalent to
\begin{align}\label{eq:exmechlam1}
    \lambda_1(t)  = 0 \quad \mbox{for all} \ t\in [t_0,t_1].
\end{align}
The adjoint equation~\eqref{eq:zeile_2} reads
\begin{align*}
    \frac{\mathrm{d}}{\mathrm{d}t}\begin{bmatrix}
        \lambda_1\\
        \lambda_2
    \end{bmatrix} &=
    \begin{bmatrix}
M^{-1} D & -M^{-1}\\
K & 0
\end{bmatrix}
\begin{bmatrix}
    \lambda_1\\
    \lambda_2
\end{bmatrix}
- \begin{bmatrix}
    M^{-1}DM^{-1} & 0 \\
    0&0
\end{bmatrix}
\begin{bmatrix}
    x_1\\
    x_2
\end{bmatrix}
\end{align*}
and inserting \eqref{eq:exmechlam1} yields
\begin{align*}
    \tfrac{\mathrm{d}}{\mathrm{d}t} {\lambda}_2(t) = 0\quad \mbox{for a.e.} \ t\in [t_0,t_1] \qquad \mathrm{and}\qquad \lambda_2(t) = -DM^{-1}x_1(t) \quad \mbox{for all} \ t\in [t_0,t_1].
\end{align*}
From this, we further conclude
$$
\tfrac{\mathrm{d}}{\mathrm{d}t}{x}_1(t)= 0\quad \mbox{for a.e.} \ t\in [t_0,t_1]\quad\text{with} \quad x_1(t)\equiv x_1^0 \qquad\text{and}\qquad x_1(t) = x_1^1 \quad \mbox{for all} \ t\in[t_0,t_1].
$$
This already yields the condition $x_1^0 = x_1^1$ on the entry and exit point of the singular arc and $\lambda_2(t) = -DM^{-1} x_1^0$ for all $t\in[t_0,t_1]$. Inserting this into the second row of the state equation yields $\tfrac{\mathrm{d}}{\mathrm{d}t}{x_2}(t) = M^{-1}x_1^0$ for a.e.\ $t\in[t_0,t_1]$ and thus
\[
x_2(t) = x_2^0 + (t-t_0)M^{-1}x_1^0\ \quad \mbox{for all}\ t\in [t_0,t_1].
\]
From this we get the additional condition $x_2^1 = x_2^0 + (t_1-t_0)M^{-1}x_1^0$, coupling the entry and exit point with the time at which the solution leaves the singular arc. In particular, if $x^0\neq x^1$, an optimal solution only exists for one particular length $t_1-t_0$ of the singular arc.

Finally, the first row of the state equation implies the formula for the optimal control, which is the \emph{feedback law}
\begin{align*}
    u(t) = DM^{-1}x_1(t) + Kx_2(t) = (D+(t-t_0)K)M^{-1}x_1^0 + Kx_2^0 \qquad \mbox{for all} \ t\in [t_0,t_1],
\end{align*}
where  the last part of the equation is an explicit solution formula in terms of the initial value.
\end{Ex}

In Example~\ref{ex:mechsys} we observed that, for second-order mechanical systems, the control on the singular arc can be expressed directly via the entry point $x^0$. The aim of this paper is to provide a generalisation of this approach to optimal control problems of the form~\eqref{eq:special_port-Hamiltonian_OCP} or, if a regularization is necessary, its regularized counterparts~\eqref{eq:perturbed_port-Hamiltonian_OCP}, or  \eqref{eq:ultimate_port-Hamiltonian_OCP}.
To derive an explicit solution formula, we recall the concept of the Drazin inverse.

\begin{Def}[{Drazin inverse, see~\cite[Definition 2.17]{KunkMehr06}}]\label{def:drazin}
Let $M\in\mathbb{K}^{n\times n}$. The unique matrix $M^D\in\mathbb{K}^{n\times n}$ with the properties
\begin{enumerate}
\item[(i)] $M^DM = MM^D$,
\item[(ii)] $M^DMM^D = M^D$,
\item[(iii)] there exists $\nu\in\mathbb{N}_0: M^DM^{\nu+1} = M^\nu$,
\end{enumerate}
is called the \emph{Drazin inverse} of $M$. 
The minimal $\nu$ satisfying (iii) is called the \emph{(Riesz) index} of $M$ and coincides with the nilpotency index of $M$, which is the size of the largest Jordan block associated with the eigenvalue zero.
\end{Def}

Recall further the well-known solution formula for initial value problems associated with linear time-invariant differential-algebraic systems with commuting coefficients using the Drazin inverse, see e.g.~\cite[Lemma 2.25 and Theorem 2.27]{KunkMehr06} for a proof.
\begin{Prop}\label{prop:representation_solution}
Let $I\subseteq\mathbb{R}$ be an interval with $t_0\in I$ and consider the  homogeneous system
\begin{align}\label{eq:general_DAE}
\left(\tfrac{\mathrm{d}}{\mathrm{d}t}E-A\right)z = 0,
\end{align}
with $E,A\in\mathbb{K}^{n\times n}$ satisfying $EA = AE$ and $\ker E \cap \ker A = \set{0}$. A function $z\in\mathcal{C}^1(I,\mathbb{R}^n)$ is a solution of \eqref{eq:general_DAE}
if, and only if, $z(t_0)\in\mathrm{im}\,E^DE$ and
\begin{align}\label{eq:solution_Drazin_inverse}
z = \exp{E^DA(\cdot-t_0)}
z(t_0).
\end{align}
\end{Prop}

Unfortunately, the matrices $\mathcal{E}$ and $\mathcal{A}_\Q$ in the optimality system \eqref{eq:the_optimality_DAE} commute if, and only if, $B = 0$ in which case the optimal control problem~\eqref{eq:ultimate_port-Hamiltonian_OCP} becomes trivial, since the boundary value problem either possesses no solution (which is generically the case) or every control is optimal. However, if the pencil $s\mathcal{E}-\mathcal{A}_\Q$ is regular (i.e. $\det (s\mathcal{E}-\mathcal{A}_\Q)\in\mathbb{K}[s]$ is not the zero-polynomial), then there exists $\mu\in\mathbb{K}$ so that $\mu\mathcal{E}-\mathcal{A}_\Q$ is invertible and the coefficient matrices of the transformed pencil $(\mu\mathcal{E}-\mathcal{A}_\Q)^{-1}(s\mathcal{E}-\mathcal{A}_\Q)$ commute as the following elementary result shows.

\begin{Lem}[{\cite[Lemma 2, p.\,418]{CampMeyeRose76}}]\label{lem:commute}
Let $E,A\in\mathbb{K}^{n\times n}$. Then for every $\mu\in\mathbb{K}$ such that $\mu E-A$ is invertible,
\begin{align*}
[(\mu E-A)^{-1}E][(\mu E-A)^{-1}A] = [(\mu E-A)^{-1}A][(\mu E-A)^{-1}E].
\end{align*}
\end{Lem}

It is then clear that the interplay of Proposition~\ref{prop:representation_solution} and Lemma \ref{lem:commute} leads to a solution formula for~\eqref{eq:the_optimality_DAE} if the pencil $s\mathcal E-\mathcal A_S$ is regular, see \cite{KunkMehr06} for a detailed analysis. However, it should be noted that this formula characterises only continuously differentiable solutions, while we initially considered weak solutions in, e.g., \eqref{eq:perturbed_port-Hamiltonian_OCP}. On the other hand, we ultimately seek for an optimal control which is a linear feedback of the state or the state and associated adjoint. In this case, the solution of the optimality system~\eqref{eq:the_optimality_DAE} is smooth and can be obtained from the solution formula. 

\section{Analysis of the optimality pencil \texorpdfstring{$\boldsymbol{s\calE-\calA_S}$}{sE-AS}}\label{sec:Optimality_pencil}

In this section, we study properties of the pencil associated to the optimality DAE~\eqref{eq:the_optimality_DAE}. In particular, we consider its regularity, which is strongly connected with the solution formula in terms of the Drazin inverse, the index and rank-minimal choice of the regularization term $\Q$. The main results of this section are summarised in the Figure~\ref{fig:Schaubild_2}.



\subsection{Regularity}\label{subsec:reg}
Conditions that characterize the regularity of the pencil  $s\mathcal E-\mathcal A_S$ via normal and staircase forms are well studied, see \cite{ByeMX07,Camp76,Meh91}.
The next proposition contains such conditions in the  special case of port-Hamiltonian systems. 
\begin{Prop}\label{prop:Campbell_for_us}
Let $(J,R,Q,B)\in\PH$ and $\Q\in\mathbb{K}^{m\times m}$ with $\Q^H = \Q\geq 0$. Then the following statements are equivalent:
\begin{enumerate}
\item[(i)] The pencil $s\mathcal{E}-\mathcal{A}_\Q$ is regular.
\item[(ii)] There exists $\mu\in\mathbb{K}\setminus(\sigma(-((J-R)Q)^H)\cup\sigma((J-R)Q))$ such that the matrix
\begin{align}\label{eq:Campbell's_qmu}
\Q_\mu := \Q - \begin{bmatrix}B^H & 0\end{bmatrix}
\begin{bmatrix}
0 & (J-R)Q - \mu I\\
((J-R)Q)^H +\mu I & QRQ
\end{bmatrix}^{-1}\begin{bmatrix}
B\\ 0
\end{bmatrix}\in\mathbb{K}^{m\times m}
\end{align}
is invertible.
\item[(iii)] There exists $\omega\in\mathbb{R}$ such that
\begin{align}\label{eq:Campbell's_port-Hamiltonian_condition}
i\omega\notin\sigma((J-R)Q)\quad\text{and}\quad\ker\Q\cap \ker RQ\big((J-R)Q-i\omega I\big)^{-1}B = \set{0}.
\end{align}
\item[(iv)] There exists $\omega\in\mathbb{R}$ such that
\begin{align}\label{eq:Campbell's_port-Hamiltonian_condition_ohne_R}
i\omega\notin\sigma(JQ)\quad\text{and}\quad\ker\Q\cap \ker RQ\big(JQ-i\omega I\big)^{-1}B = \set{0}.
\end{align}
\item[(v)] There exists $\omega\in\mathbb{R}$ such that
\begin{align}
    i\omega\notin\sigma(JQ)\quad\text{and}\quad\ker\Q\cap B^{-1}(JQ-i\omega I)\ker RQ = \set{0},
\end{align}
\end{enumerate}
\end{Prop}
\begin{proof}
The equivalence of (i)--(iii) follows directly from~Proposition 1 and Theorem 5 from~\cite{Camp76}, and the equivalence of (iv) and (v) is straightforward. It remains to prove that $s\mathcal{E}-\mathcal{A}_\Q$ is regular if, and only if, (iv) holds. To this end, we first show that for all $z\in\C\backslash\big(\sigma(JQ)\cup\sigma\big((J-R)Q\big)\big)$ we have
\begin{align*}
\big(JQ-z I
\big)^{-1}RQ\big((J-R)Q-z I
\big)^{-1} = \big((J-R)Q-z I%
\big)^{-1}RQ\big(JQ-z I
\big)^{-1}.
\end{align*}
This follows immediately by applying $\big(JQ-z I\big)^{-1}$ from the left and from the right to the equality
\begin{align*}
RQ\big((J-R)Q-z I
\big)^{-1}\big(JQ-z I
\big) &  = RQ(JQ-RQ-z I
)^{-1}(JQ-z I
)\\
& = RQ(JQ-RQ-z I
)^{-1}(JQ-RQ-z I
+RQ)\\
& = RQ+RQ(JQ-RQ-z I
)^{-1}RQ\\
& = (JQ-RQ-z I
+RQ)(JQ-RQ-z I
)^{-1}RQ\\
& = (JQ-z I)(JQ-RQ-z I
)^{-1}RQ\\
& = \big(JQ-z I
\big)\big((J-R)Q-z I
\big)^{-1}RQ.
\end{align*}
Hence, for all $z\in\mathbb{C}\setminus\big(\sigma (JQ)\cup\sigma\big((J-R)Q\big)\big)$,
\begin{align*}
\ker RQ((J-R)Q-z I)^{-1}B &= \ker \big((J-R)Q-z I)\big)^{-1}RQ(JQ-z I)^{-1}B\\
&= \ker RQ(JQ-z I)^{-1}B.
\end{align*}
Now, note that if~\eqref{eq:Campbell's_port-Hamiltonian_condition} 
holds for some $\omega\in\mathbb R$, then it holds in a neighborhood of this $\omega$. The same is true for~\eqref{eq:Campbell's_port-Hamiltonian_condition_ohne_R}. 
Since the spectrum of a matrix is finite, this proves the claim.
\end{proof}

\begin{Rem}
In \cite[Theorem 7]{Camp76} it is proven that regularity of $s\calE - \calA_S$ is also equivalent to the uniqueness of optimal controls for the OCP \eqref{eq:ultimate_port-Hamiltonian_OCP}, if they exist.
\end{Rem}

Typically, the necessary and sufficient conditions for the regularity of the pencil $s\mathcal{E}-\mathcal{A}_\Q$ in Proposition~\ref{prop:Campbell_for_us}  are hard to verify in applications, since all of them are existence statements which contain resolvents.
Therefore, we seek to find explicit or resolvent-free conditions.
A \textit{necessary} condition, 
that meets both requirements is the following result inspired by~\cite[Proposition 6]{Camp76}.

\begin{Prop}\label{prop:Campbell's_mistake}
Let $(J,R,Q,B)\in\PH$ and $\Q\in\mathbb{K}^{m\times m}$ with $\Q^H = \Q\geq 0$. If $s\mathcal{E}-\mathcal{A}_\Q$ is regular, then 
\begin{align}\label{eq:equation_5}
\ker \Q\,\cap\,\bigcap_{r = 0}^{n-1} \ker RQ(JQ)^rB = \set{0}.
\end{align}
The converse implication holds for single-input-single-output systems (i.e. the case $m = 1$), but not in general.
\end{Prop}
\begin{proof}
If $s\mathcal{E}-\mathcal{A}_\Q$ is regular, then~\eqref{eq:Campbell's_port-Hamiltonian_condition_ohne_R} holds for all but finitely many $\omega\in\mathbb{R}$ due to Proposition~\ref{prop:Campbell_for_us}\,(iv), and since the set of all $\mu\in\mathbb{K}$ so that $\mu\mathcal{E}-\mathcal{A}_\Q$ is invertible
is either empty or co-finite. 
Thus, there exists some $\Omega>\|JQ\|$ so that for all $\omega\in\mathbb{R}\setminus(-\Omega,\Omega)$ one has that $i\omega\notin\sigma(JQ)$ and $\ker\Q\cap\ker RQ(JQ-i\omega I)^{-1}B = \set{0}$.
Hence, for all $\omega\in\mathbb{R}\setminus(-\Omega,\Omega)$ the Neumann series 
\[
(JQ-i\omega I)^{-1} = -\sum_{k = 0}^\infty\left(\frac{-i}{\omega}\right)^{k+1}(JQ)^k
\]
converges absolutely and  
\begin{align}\label{eq:Neumann}
RQ(JQ-i\omega I)^{-1}B = 
-\sum_{k = 0}^\infty\left(\frac{-i}{\omega}\right)^{k+1} RQ(JQ)^kB.
\end{align}
Using the  Cayley-Hamilton Theorem,~see e.g. \cite[Theorem 8.6]{LieM15a}, we have that for all $r\geq n$ there exist coefficients $a_0^r,\ldots,a_{n-1}^r\in\mathbb{K}$, such that 
\[
(JQ)^r = \sum_{j = 0}^{n-1} a_j^r(JQ)^j.
\]
Hence, we conclude that for all $r\geq n$
\begin{align*}
\bigcap_{\rho = 0}^{n-1}\ker RQ(JQ)^\rho B\,\subset\, \ker \sum_{\rho = 0}^{n-1} a_\rho^r RQ(JQ)^\rho B = \ker RQ(JQ)^rB.
\end{align*}
This shows that $\ker\Q\cap\bigcap_{r = 0}^{n-1}\ker RQ(JQ)^rB = \ker\Q\cap\bigcap_{r = 0}^{\infty}\ker R(JQ)^rB$. Therefore,  for all $v\in\ker\Q\cap\bigcap_{r = 0}^{n-1}\ker RQ(JQ)^rB$ and for all 
$\omega\in\mathbb{R}\setminus(-\Omega,\Omega)$, we have that 
\begin{align*}
RQ(JQ-i\omega I)^{-1}Bv = -\sum_{k = 0}^\infty RQ(JQ)^kBv\left(\frac{-i}{\omega}\right)^{k+1} = 0
\end{align*}
and hence
$v = 0$. 

Consider the case that $m=1$, i.e. $\Q\in\mathbb{K}^{1\times 1}$ 
and $B\in\K^{n\times 1}$.
If $\Q$ is non-zero, then the pencil $s\mathcal{E}-\mathcal{A}_\Q$ is obviously regular. If $\Q=0$ then
Proposition~\ref{prop:Campbell_for_us}\,(iv) yields that $s\mathcal{E}-\mathcal{A}_\Q$ is regular if, and only if, there exists $\omega\in\mathbb{R}$, with
$i\omega\notin\sigma(JQ)$ and $RQ(JQ-i\omega I)^{-1}B\neq 0$.
Seeking a contradiction, assume that~\eqref{eq:equation_5} holds and that $RQ(JQ-i\omega I)^{-1}B = 0$ for all $\omega\in\mathbb{R}$ with $i\omega\notin\sigma(JQ)$. In this case, the representation~\eqref{eq:Neumann} and the identity theorem for power series \cite[Corollaries 1.2.4 and 1.2.7]{KranPark02} yield that $RQ(JQ)^rB = 0$ for all $r\in\mathbb{N}_0$, which
contradicts our assumption. Therefore $s\mathcal{E}-\mathcal{A}_\Q$ is indeed regular.

The following example shows that the converse implication does not hold in general.
Let $n = m = 2$, $Q = B = I_2$, $\Q = 0_{2\times 2}$ and
\begin{align*}
J = \begin{bmatrix}
0 & -1\\
1 & 0
\end{bmatrix},\qquad R = \begin{bmatrix}
1 & 0\\
0 & 0
\end{bmatrix}.
\end{align*}
Then  for 
$\omega\in\mathbb{R}\setminus\set{\pm 1}$
we have
\begin{align*}
\ker RQ(JQ-i\omega I)^{-1} = (J-i\omega I)\ker R = \set{(z,i\omega z)^\top\,\big\vert\,z\in\mathbb{C}}\neq \set{0}.
\end{align*}
Since $\ker\Q = \mathbb{C}^m$, Proposition~\ref{prop:Campbell_for_us}\,(iv) yields that $s\mathcal{E}-\mathcal{A}_\Q$ is singular. On the other hand, we have, for all $\omega,\mu\in\mathbb{R}\setminus\set{\pm 1}$ with $\omega\neq \mu$ 
\begin{align*}
\ker R(J-i\omega I)^{-1}\cap \ker R(J-i\mu I)^{-1} = 
\set{0}.
\end{align*}
Using the Neumann series, 
this yields that~\eqref{eq:equation_5} holds; for details see the upcoming Lemma~\ref{rem:upcoming}. This shows that~\eqref{eq:equation_5} is, in general, not sufficient for regularity of $s\mathcal{E}-\mathcal{A}_\Q$.
\end{proof}

We have seen that the condition~\eqref{eq:equation_5} is necessary but not sufficient for regularity of $s\mathcal{E}-\mathcal{A}_\Q$. The following lemma reveals the underlying reason for this shortcoming: $s\mathcal{E}-\mathcal{A}_S$ is regular if, and only if, $\ker S\cap \ker RQ(JQ-i\omega I)^{-1}B = \set{0}$ for one fixed $\omega\in\mathbb{R}$, whereas the condition~\eqref{eq:equation_5} holds if and only if the intersection of the kernel of $RQ(JQ-i\omega I)^{-1}B$ over infinitely many $\omega\in\mathbb{R}$ has trivial intersection with $\ker S$. 

\begin{Lem}\label{rem:upcoming}
Let $(J,R,Q,B)\in\PH$ and $\Q\in\mathbb{K}^{m\times m}$ with $\Q^H = \Q\geq 0$. Then~\eqref{eq:equation_5} holds if, and only if, there exists $\Omega>0$ such that all $ \omega_1>\omega_0>\Omega$ fulfil the property
\begin{align}\label{eq:Ringstuff}
\ker S\ \cap \bigcap_{\omega_0<\omega<\omega_1}\ker RQ(JQ-i\omega I)^{-1}B = \set{0}.
\end{align}
\end{Lem}
\begin{proof}
Let $\phi := \frac{1}{\omega}$. Then, the series~\eqref{eq:Neumann} is a formal power series $F\in\mathbb{K}^{n\times m}[[\phi]]$ with non-vanishing convergence radius $\rho\geq \frac{1}{\norm{JQ}}$, where we use the convention $\frac{1}{0} = \infty$. Each $v\in\mathbb{K}^m$ induces then the formal power series $F(\cdot)v\in\mathbb{K}^n[[\phi]]$ whose convergence radius is at least $\rho$. By definition of $F$, we then have for all $\phi\in (-\rho,\rho)$ and for all $ v\in\mathbb{K}^m$ that
\begin{align*}
 F(\phi)v =  RQ\left(JQ-\frac{i}{\phi} I\right)^{-1}Bv.
\end{align*}
Therefore, the identity theorem for power series \cite[Corollaries 1.2.4 and 1.2.7]{KranPark02} yields the equivalence
\begin{align*}
v\in\ker \Q\cap\bigcap_{r = 0}^{n-1} \ker RQ(JQ)^rB \iff v\in\ker S~\text{and}~F(\cdot)v\equiv 0.
\end{align*}
The latter is equivalent to
\begin{align*}
v\in \ker S\ \cap \bigcap_{\omega_0<\abs{\omega}<\omega_1}\ker RQ(JQ-i\omega I)^{-1}B
\end{align*}
for all $\omega_0<\omega_1\in (1/\rho,\infty)$.
This shows that~\eqref{eq:equation_5} and~\eqref{eq:Ringstuff} are indeed equivalent. 
\end{proof}

As we have demonstrated, the inverse-free necessary condition~\eqref{eq:equation_5} is not sufficient for regularity of $s\mathcal{E}-\mathcal{A}_S$. However, 
the next proposition presents an inverse-free necessary and sufficient condition.

\begin{Prop}\label{prop:final_equivalence}
Let $(J,R,Q,B)\in\PH$ and $\Q\in\mathbb{K}^{m\times m}$ with $\Q^H = \Q\geq 0$. The pencil $s\mathcal{E}-\mathcal{A}_\Q$ is regular if, and only if, there exists $\omega\in\mathbb{R}$ such that 
\begin{align}\label{eq:some_kerR_imB_condition}
i\omega\notin\sigma(JQ)\quad\text{and}\quad B\ker \Q\cap (JQ-i\omega I)\ker RQ = \set{0} = \ker B\cap\ker\Q.
\end{align}
\end{Prop}
\begin{proof}
Assume that the pencil $s\mathcal{E}-\mathcal{A}_\Q$ is regular. By Proposition~\ref{prop:Campbell_for_us}\,(iv), there exists some $\omega\in\mathbb{R}$ with $i\omega\notin\sigma(JQ)$ such that $\ker\Q\cap\ker RQ(JQ-i\omega I)^{-1}B = \set{0}$. In particular, this implies that $\ker\Q\cap\ker B = \set{0}$. Let $v\in B\ker \Q\cap (JQ-i\omega I)\ker RQ$. Then, there exists $u\in\ker\Q$ so that $v = Bu$ and $(JQ-i\omega I)^{-1}Bu\in\ker RQ$ or, equivalently, $u\in \ker RQ(JQ-i\omega I)^{-1}B$. Thus, $u = 0$ and hence $v = 0$, so that~\eqref{eq:some_kerR_imB_condition} holds.

Conversely, let $\omega\in\mathbb{R}$ be such that~\eqref{eq:some_kerR_imB_condition} holds for this particular $\omega$. Let $u\in\ker\Q\cap\ker RQ(JQ-i\omega I)^{-1}B$ and set $v:=Bu\in B\ker\Q$. Then $RQ(JQ-i\omega I)^{-1}v = 0$, i.e. $v\in (JQ-i\omega I)\ker RQ$. Hence, we obtain $v = 0$ or, equivalently, $u\in\ker B$. Since $\ker B\cap\ker \Q = \{0\}$, we conclude that $u = 0$ and Proposition~\ref{prop:Campbell_for_us}\,(iv) yields that $s\mathcal{E}-\mathcal{A}_\Q$ is regular.
\end{proof}




Condition~\eqref{eq:some_kerR_imB_condition} implies, in particular, that the block matrix $[B,JQ-i\omega I]$ acts injectively on the cartesian product $\ker\Q\times\ker RQ$. This immediately yields a rank criterion for the regularity of $s\mathcal{E}-\mathcal{A}_\Q$.

\begin{Cor}\label{cor:rank_criterion}
Let $(J,R,Q,B)\in\PH$ and $\Q\in\mathbb{K}^{m\times m}$ with $\Q = \Q^H\geq 0$. Then the pencil $s\mathcal{E}-\mathcal{A}_\Q$ is regular if, and only if, there exists $\omega\in \mathbb R$ such that
\begin{align}\label{eq:condition_1}
\dim\big(B\ker \Q + (JQ-i\omega I)\ker RQ\big) = m+n-(\rk \Q + \rk RQ). 
\end{align}
\end{Cor}
\begin{proof}
Let $Z_\Q\in\mathbb{K}^{m\times (m-\rk \Q)}$ and $Z_R\in\mathbb{K}^{n\times (n-\rk RQ)}$ be matrices with full column rank 
such that $\im Z_\Q = \ker \Q$ and $\im Z_R = \ker RQ$
and we have the formula\[
\dim (B\ker \Q + (JQ-i\omega I)\ker RQ) = \rk[BZ_\Q,(JQ-i\omega I)Z_R].
\]
Let $s\mathcal{E}-\mathcal{A}_\Q$ be regular. In view of Proposition~\ref{prop:final_equivalence}, this is the case if, and only if, the restriction of $B$ to $\ker \Q$ is injective and $B\ker \Q\cap (JQ-i\omega I)\ker RQ  =\set{0}$ for some $\omega\in\mathbb{R}$ so that $JQ-i\omega I$ is invertible. Hence, for this particular $\omega$ we have $\rk BZ_\Q = \dim B\ker \Q = \dim\ker \Q = m-\rk \Q$ and $\rk (JQ-i\omega I)Z_R = \dim (JQ-i\omega I)\ker RQ = \dim\ker RQ = n-\rk RQ$ and the sum $B\ker \Q + (JQ-i\omega I)\ker RQ$ is a direct sum. Thus, \eqref{eq:condition_1} holds, i.e.
\begin{align*}
\dim (B\ker \Q + (JQ-i\omega I)\ker RQ)
& = \dim(B\ker \Q) + \dim((JQ-i\omega I)\ker RQ)\\
& = n+m-(\rk \Q+\rk RQ).
\end{align*}
Conversely, assume that~\eqref{eq:condition_1} is satisfied. The set of all $\omega\in\mathbb{R}$ such that $\rk[BZ_\Q,(JQ-i\omega I)Z_R]<n+m-(\rk \Q+\rk RQ)$ 
either coincides with $\mathbb{R}$ or is finite. By condition~\eqref{eq:condition_1}, the former does not hold.
In particular, there is some $\omega\in\mathbb{R}$ so that $i\omega\notin\sigma(JQ)$ such that~\eqref{eq:condition_1} holds at this $\omega$. Then, we can conclude analogously that~\eqref{eq:some_kerR_imB_condition} holds. From Proposition~\ref{prop:final_equivalence}, we conclude that the pencil $s\mathcal{E}-\mathcal{A}_\Q$ is regular.
\end{proof}
Whereas the previous characterizations of regularity required the choice $i\omega\notin\sigma(JQ)$, condition \eqref{eq:condition_1} can be verified without computing the spectrum of $JQ$. 
\subsection{Rank-minimal regularization of the cost functional}\label{subsec:minreg}
Recall that one of our goals is to choose a positive semidefinite matrix $\Q\in\K^{m\times m}$ of minimal rank, 
to achieve
regularity of the optimality system. In the following we 
construct a regularization term of minimal rank in the cost functional of \eqref{eq:ultimate_port-Hamiltonian_OCP} such that the resulting pencil in the optimality system is regular.

\begin{Lem}\label{lem:rank_minimal_realization}
Let $(J,R,Q,B)\in\PH$ be such that the pencil $s\mathcal{E}-\mathcal{A}_S$ is singular for $S=0$. Let 
$\omega\in\mathbb{R}$ with $i\omega\notin\sigma(JQ)$ and consider 
the subspace
\begin{align*}
V := B^{-1}(JQ-i\omega I)\ker RQ.
\end{align*}
Let $\Q\in\mathbb{K}^{m\times m}$ with $\Q = \Q^H\geq 0$ and $\ker \Q = V^\bot$ (e.g. the orthogonal projection onto $V$). Then $s\mathcal{E}-\mathcal{A}_\Q$ is a regular pencil. Moreover, if $\omega$ is chosen such that
\begin{align}\label{eq:dim_minimal}
    \dim B^{-1}(JQ-i\omega I)\ker RQ = \min\set{\dim B^{-1}(JQ-z I)\ker RQ\,\big\vert\, z\in\mathbb{C}},
\end{align}
which holds for all but finitely many $\omega\in\mathbb{R}$ with $i\omega\notin\sigma(JQ)$, then
\begin{align}\label{eq:rank_minimal}
\rk \Q = \min\set{\rk\widehat{\Q}\,\left\vert\,s\mathcal{E}-\mathcal{A}_{\widehat{\Q}}~\text{regular}\right.}.
\end{align}
\end{Lem}
\begin{proof}
Since $V\cap V^\bot = \set{0}$, Proposition~\ref{prop:Campbell_for_us}\,(v) yields that $s\mathcal{E}-\mathcal{A}_\Q$ is regular. Assume now that~\eqref{eq:dim_minimal} holds. We show that $\Q$ is rank-minimal in the sense of~\eqref{eq:rank_minimal}. Let $\widehat{\Q}\in\mathbb{K}^{m\times m}$ with $\rk\widehat{\Q}\leq\rk\Q$ so that $\widehat{\Q} = \widehat{\Q}^H\geq 0$ and $s\mathcal{E}-\mathcal{A}_{\widehat{\Q}}$ is regular. By Proposition~\ref{prop:Campbell_for_us}\,(v), we find some $\omega_0\in\mathbb{R}$ so that $i\omega_0\notin\sigma(JQ)$ and $\ker\widehat{\Q}\cap B^{-1}(JQ-i\omega_0 I)\ker RQ = \set{0}$. Since~\eqref{eq:dim_minimal} holds if, and only if, 
\begin{align*}
    \dim \left(\mathrm{im}\,B + (JQ-i\omega I)\ker RQ\right) = \max\set{\dim\left(\mathrm{im}\,B + (JQ-z I)\ker RQ\right)\,\big\vert\,z\in\mathbb{C}},
\end{align*}
the condition~\eqref{eq:dim_minimal} is fulfilled for all but finitely many $\omega\in\mathbb{R}$. Therefore, we may choose $\omega_0$ so that $\dim B^{-1}(JQ-i\omega_0 I)\ker RQ = \dim V$. We conclude that
\begin{align*}
m & \geq \dim \big(\ker\widehat{\Q} + B^{-1}(JQ-i\omega_0 I)\ker RQ\big)\\
& = \dim\ker\widehat{\Q} + \dim B^{-1}(JQ-i\omega_0 I)\ker RQ\\
& \geq \dim\ker\Q+\dim V = \dim V^\bot + \dim V = m.
\end{align*}
Therefore the inequalities are equalities and we conclude $\dim\ker\widehat{\Q} = \dim\ker\Q$ or, equivalently, $\rk\widehat{\Q} = \rk\Q$. Thus, $\Q$ is indeed rank-minimal in the sense of~\eqref{eq:rank_minimal}.
\end{proof}

The following result extends \cite[Theorem 8]{FaulMascPhilSchaWort21} to regularized cost functionals and not necessarily invertible matrices $Q$, providing a sufficient condition for regularity of $s\mathcal{E}-\mathcal{A}_\Q$.

\begin{Prop}\label{prop:Schaller_Philipp_für_uns}
Let $(J,R,Q,B)\in\PH$ and $\Q\in\mathbb{K}^{m\times m}$ with $\Q = \Q^H\geq 0$. If
\begin{align}\label{eq:equation_2}
B\ker\Q\cap\ker RQ = \set{0} = \ker B\cap\ker\Q,
\end{align}
then $s\mathcal{E}-\mathcal{A}_\Q$ is regular. The converse is false, in general. Moreover, if \eqref{eq:equation_2} holds, then optimal controls for the OCP \eqref{eq:ultimate_port-Hamiltonian_OCP} (if they exist) are unique.
\end{Prop}
\begin{proof}
In the following, denote by $M_{S,T}(A)$ the minor of a matrix $A$ with row and column index sets $S$ and $T$, respectively, where $|S|=|T|$.

Let $Z_\Q\in\mathbb{K}^{m\times (m-\rk \Q)}$ and $Z_R\in\mathbb{K}^{n\times (n-\rk RQ)}$ such that $\im Z_Q = \ker \Q$ and $\im Z_R = \ker RQ$. Then assumption~\eqref{eq:equation_2} yields that $[BZ_\Q,Z_R]$ has full column rank $\tau = n+m-(\rk\Q+\rk RQ)$. Hence, there exists a non-zero minor $M_{S,T}([BZ_\Q,Z_R])$ of $[BZ_\Q,Z_R]$ with $|S|=|T|=\tau$ . By continuity, there exists $\varepsilon>0$ such that $M_{J,K}([BZ_\Q,(irJQ+I)Z_R])$ does not vanish for $r\in (0,\varepsilon)$. Hence,
\[
\tau = \rk [BZ_\Q,(i\tfrac{\varepsilon}2JQ+I)Z_R] = \rk\big[BZ_\Q,(JQ-i\tfrac{2}{\varepsilon}I)Z_R\big] = \dim(B\ker S + (JQ-i\tfrac{2}{\varepsilon}I)\ker RQ),
\]
which is \eqref{eq:condition_1} with $\omega = \tfrac{2}{\varepsilon}$.

To show that the converse is, in general, not true, consider $n = 2$ and $m = 1$. Let $\Q = 0$ and let $B = \sbvec 10$, which certainly fulfills $\ker B\cap\ker S = \set{0}$. Further, set
\begin{align*}
J = \begin{bmatrix}
0 & 1\\
-1 & 0
\end{bmatrix},\qquad
R = \begin{bmatrix}
0 & 0\\
0 & 1
\end{bmatrix},
\qquad\text{and}\qquad 
Q = I.
\end{align*}
Then, $\ker R = \im B\neq \set{0}$ and \eqref{eq:equation_2} is violated. For all $z\in\mathbb{R}$ and $v = v_1\sbvec 10\in\ker R$ we have 
\begin{align*}
-v_1\bvec{i\omega}{1} = (J-i\omega I)v\in\im B
\end{align*}
if, and only if, $v = 0$ and hence we conclude that~\eqref{eq:some_kerR_imB_condition} holds. Proposition~\ref{prop:final_equivalence} yields that $s\mathcal{E}-\mathcal{A}_\Q$ is regular.
\end{proof}
The previous result included a condition which is sufficient but not necessary for regularity. In the following we show that it is indeed necessary for regularity with an additional condition on the index.

\subsection{Index of the optimality DAE \eqref{eq:the_optimality_condition}}\label{subsec:index}
An important property for the numerical solution of the optimality boundary value problem~\eqref{eq:the_optimality_condition} is the 
\emph{Kronecker index}, i.e. the size of the largest block associated with the eigenvalue $\infty$ in the Kronecker canonical form \cite{TheorMatr} of the pencil $s\calE  - \calA_S$. The next proposition characterizes condition \eqref{eq:equation_2} in the case $S=0$ in terms of the Kronecker index.

\begin{Prop}\label{prop:indexthree}
Let $(J,R,Q,B)\in\PH$,
and $S = 0_{m\times m}$. Then
\begin{equation}\label{e:spfwm2}
\im B\cap\ker RQ = \{0\} = \ker B
\end{equation}
holds if, and only if, the pencil $s\calE  - \calA_S$ is regular with Kronecker index three.
\end{Prop}
\begin{proof}
First of all, since $S=0$ we have that \eqref{e:spfwm2} is equivalent to \eqref{eq:equation_2}. Hence, by Proposition~\ref{prop:Schaller_Philipp_für_uns}, the pencil is regular in both logical directions of the proof. Therefore, we may assume that the pencil is regular. By Proposition~\ref{prop:final_equivalence}, there exists $\omega\in\R$ with $i\omega\notin\sigma(JQ)\cap\sigma(-JQ)$ such that
$$
\im B\cap(JQ-i\omega I)\ker RQ = \im B\cap(JQ+i\omega I)\ker RQ = \{0\} = \ker B.
$$
We now have to prove that $V := \im B\cap\ker RQ = \{0\}$ if, and only if, the pencil is of index three. We set $A := (J-R)Q$.

One easily computes the first four elements of the increasing \emph{Wong sequence} (see~\cite[Definition 2.1]{BergIlchTren12} or \cite{Wong74}), defined by $\mathcal{W}_0 = \{0\}$ and $\mathcal{W}_{k+1} = \calE^{-1}\calA_S \mathcal{W}_k$ for $k\in\mathbb{N}_0$, as $\mathcal{W}_1 = \{0\}\times\{0\}\times\K^m$, $\mathcal{W}_2 = \{0\}\times\mathrm{im}\, B\times\K^m$,
\begin{align*}
\mathcal{W}_3 & = \left\{\bvec{-QRQv}{Av+w} : v,w\in\mathrm{im}\, B\right\}\times\K^m\\
\mathcal{W}_4 & = \left\{\bvec{-QRQJQv-QRQw}{A^2v+Aw+y} : v\in V,\,w,y\in\mathrm{im}\, B\right\}\times\K^m.
\end{align*}
Note that here $\mathcal E^{-1}$ denotes the preimage of the map defined by $\mathcal E$ and not the inverse of $\mathcal E$.
Hence, $\mathcal{W}_0\neq \mathcal{W}_1$ and $\mathcal{W}_1\neq \mathcal{W}_2$. Moreover, $\mathcal{W}_2\neq \mathcal{W}_3$, since $\mathcal{W}_2 = \mathcal{W}_3$ is equivalent to $\im B\subset\ker RQ$ and $JQ\,\im B\subset\im B$, which leads to
$$
(JQ-i\omega I)\im B\subset\im B\cap (JQ-i\omega I)\ker R = \{0\},
$$
i.e., $B=0$, contradicting $\ker B=\{0\}$. From~\cite[Proposition 2.10]{BergIlchTren12}, we conclude that the Kronecker index of the pencil index is at least three, and that the index is  three if, and only if, $\mathcal{W}_3 = \mathcal{W}_4$.

Now, $\mathcal{W}_3=\mathcal{W}_4$ is easily seen to be equivalent to the statement that for each $x\in V$ there exist $v,w\in\im B$ such that $RQJQx = RQv$ and $JQJQx = JQv + w$. Hence, $V=\{0\}$ obviously implies $\mathcal{W}_3 = \mathcal{W}_4$. Conversely, assume that $\mathcal{W}_3 = \mathcal{W}_4$ and let $x\in V$.
Then there exist $v,w\in\im B$ such that $RQJQx = RQv$ and $(JQ)^2x = JQv + w$. Since $RQx=0$ and thus $RQ[(JQ+i\omega I)x - v]=0$, we have
\begin{align*}
(JQ-i\omega I)[(JQ+i\omega I)x - v]
&= ((JQ)^2+\omega^2 I)x - (JQ-i\omega I)v\\
&= w + \omega^2x + i\omega v\in\im B\cap (JQ-i\omega I)\ker RQ = \{0\},
\end{align*}
and hence
$$
(JQ+i\omega I)x = v\in \im B\cap (JQ+i\omega I)\ker RQ = \{0\}.
$$
Thus, $x=0$ follows which  shows that $V = \{0\}$.
\end{proof}

We summarize our results and we illustrate further equivalences and implications in Figure \ref{fig:Schaubild_2}.
\newpage

\ \vspace*{2cm}
\begin{figure}[H]
\centering
 \resizebox{.8\textwidth}{!}{\begin{tikzpicture}[node distance = 2cm, auto, paths/.style={->, thick, >=stealth'}, decoration={markings,
	mark=between positions 0 and 1 step 6pt
	with { \draw [fill] (0,0) circle [radius=1pt];}}]
\hspace{0.8cm} 
\node [blocktest] (pert) {$s\begin{bmatrix}0& I&0 \\-I&0& 0\\0& 0&0\end{bmatrix} - \begin{bmatrix}0 & (J-R)Q & B\\((J-R)Q)^H & QRQ & 0\\ B^H & 0 & \Q \end{bmatrix}$ is regular};
\node [block7,below of=pert, node distance=3.5cm] (cent) {\centering There exists $\mu\in\mathbb{C}$ such that\\[.1cm]$\Q - \begin{bmatrix}B^H & 0\end{bmatrix} \begin{bmatrix}
0 & \mu I-(J-R)Q \\
 -\mu I -((J-R)Q)^H & QRQ
\end{bmatrix}^{-1}\begin{bmatrix}
B\\ 0
\end{bmatrix}$ is invertible
};
\node [block7,below of=cent, node distance=3.5cm] (hier) {\centering There exists $\omega\in\mathbb{R}$ such that\\[.1cm] $ i\omega\notin\sigma((J-R)Q)$ and $\ker\Q\,\cap\ker RQ((J-R)Q-i\omega)^{-1}B = \set{0}$};
\node [block3,below of=hier, node distance=3.5cm] (dort) {\centering  There exists $\omega\in\mathbb{R}$ such that\\[.1cm] $i\omega\notin\sigma(JQ)$ and $\ker \Q\,\cap\ker RQ(JQ-i\omega I)^{-1}B = \set{0}$};
\node [block7,below of=dort, node distance=3.5cm] (dortter) {\centering  There exists $\omega\in\mathbb{R}$ such that\\[.1cm] $ i\omega\notin\sigma(JQ)$ and $ B \ker \Q \cap (JQ-i\omega I)\ker RQ = \set{0} = \ker B\cap\ker \Q$};
\node [block3,below of=dortter, node distance=3.5cm] (dortt) {\centering  There exists $\omega\in\mathbb{R}$ such that\\[.1cm]$ i\omega\notin\sigma(JQ)$ and $\ker\Q \cap B^{-1}(JQ-i\omega I)\ker RQ = \set{0}$};
\node [below of = dortt, node distance = 3.5cm] (centt) {};
\node [block4, right of = centt, node distance = 4.5cm] (rechts) {$\ker \Q\ \cap \bigcap\limits_{r = 0}^n \ker RQ(JQ)^rB = \set{0}$};
\node [block5, left of = centt, node distance = 4.5cm] (links) {$B\ker \Q\cap\ker RQ = \set{0} = \ker B \cap \ker \Q$};

\draw [double,<->] (pert) --  node [] {Proposition~\ref{prop:Campbell_for_us}\,(ii) and \cite[Proposition 1]{Camp76}}(cent);
\draw [double,<->] (cent) --  node [] {Proposition~\ref{prop:Campbell_for_us}\,(iii) and \cite[Theorem 5]{Camp76}}(hier);
\draw [double,<->] (hier) --  node [] {Proposition~\ref{prop:Campbell_for_us}\,(iv)}(dort);
\draw [double,<->] (dort) --  node [] {Proposition~\ref{prop:final_equivalence}}(dortter);
\draw [double,<->] (dortter) --  node [] {Proposition~\ref{prop:Campbell_for_us}\,(v)}(dortt);
\begin{scope}[transform canvas = {xshift = -0.5cm}]
\draw [double,<-] (dortt) -- node [] {Proposition~\ref{prop:Schaller_Philipp_für_uns}} (links);
\end{scope}
\begin{scope}[transform canvas = {xshift = +1.5cm}]
\draw [double,->] (dortt) -- node [] {Proposition~\ref{prop:Campbell's_mistake}} (rechts);
\end{scope}
\begin{scope}[transform canvas = {xshift = 0.5cm}]
\draw [double,<-] (dortt) -- node [anchor=east] {} coordinate (m0) (rechts);
\draw [black] ($(m0)+(0.1,0.15)$) -- ($(m0)+(-0.1,-0.15)$);
\end{scope}
\begin{scope}[transform canvas = {xshift = -1.5cm}]
\draw [double,->] (dortt) -- node [anchor=east] {} coordinate (m) (links);
\draw [black] ($(m)+(0.1,-0.15)$) -- ($(m)+(-0.1,0.15)$);
\end{scope}
\end{tikzpicture}}
\caption{Equivalent, sufficient and necessary conditions for for regularity of the optimality DAE~\eqref{eq:the_optimality_condition}.}
\label{fig:Schaubild_2}
\end{figure}
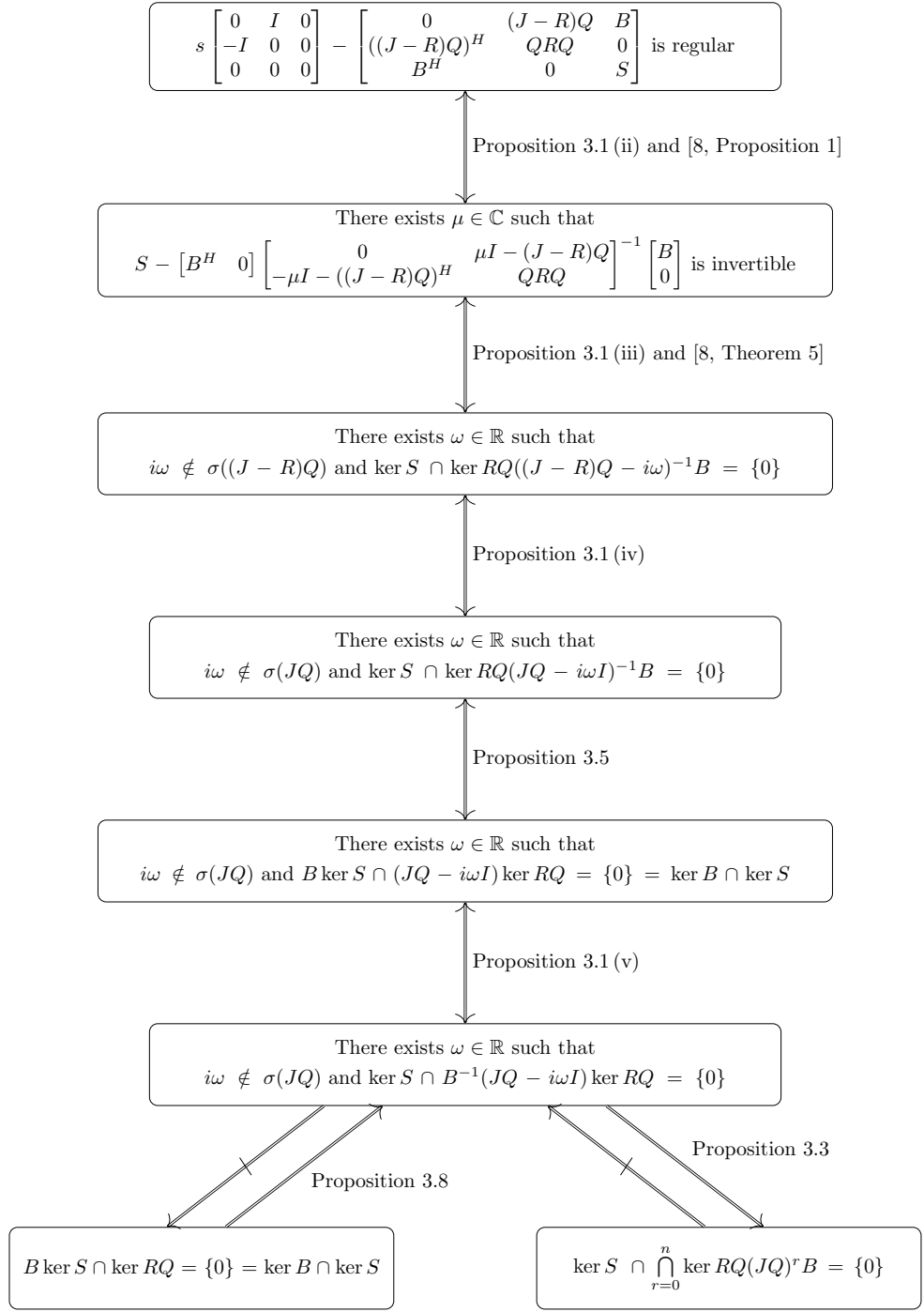

\newpage


\section{Existence of an optimal trajectory}\label{sec:existence}

In our considerations so far, we have only focused on the regularity of the pencil associated with the necessary and sufficient optimality conditions of the port-Hamiltonian optimal control problem~\eqref{eq:perturbed_port-Hamiltonian_OCP}. Using the general solution formula that we presented in Proposition~\ref{prop:representation_solution}, we can derive a solution formula for~\eqref{eq:the_optimality_DAE} along the lines of \cite[pp.\,1096]{Camp76}.

Recall the Drazin inverse from Definition~\ref{def:drazin} and note that $M^DM = MM^D$ coincides with the spectral projection onto the direct sum of the generalized eigenspaces (root subspaces) of $M$ corresponding to its non-zero eigenvalues.

\begin{Prop}\label{prop:initial_conditions}
Let $(J,R,Q,B)\in\PH$ and $\Q\in\mathbb{K}^{m\times m}$ with $\Q^H = \Q\geq 0$ be such that $s\mathcal{E}-\mathcal{A}_\Q$ is regular. Choose $\mu\in\mathbb{K}\setminus(\sigma((J-R)Q)\cup \sigma(-((J-R)Q)^H))$ such that $\mu\mathcal{E}-\mathcal{A}_\Q$ is invertible, define $\Q_\mu$ as in~\eqref{eq:Campbell's_qmu} and
\begin{align}
E_\mu & := \begin{bmatrix}
0 & (J-R)Q - \mu I \\
((J-R)Q)^H+\mu I & QRQ
\end{bmatrix}^{-1}
,\nonumber\\
N_\mu & := \left(E_\mu + E_\mu\begin{bmatrix}
B\\ 0
\end{bmatrix}\Q_\mu^{-1}\begin{bmatrix}B^H &0\end{bmatrix}E_{\mu}\right)\bmat 0{-I}I0,\label{eq:Nmu}\\
M_\mu & := \Q_\mu^{-1}\begin{bmatrix}B^H &0\end{bmatrix}E_\mu\bmat 0{-I}I0.\label{eq:Mmu}
\end{align}
A function $\begin{bmatrix}\lambda^T & x^T & u^T\end{bmatrix}^T\in\mathcal{C}^1([t_0,t_1],\mathbb{K}^{n+n+m})$ is a solution of~\eqref{eq:the_optimality_DAE} if, and only if, $\begin{bmatrix}
   \la(t_0)\\ x(t_0)) 
\end{bmatrix}\in\im(N_\mu^DN_\mu)$ and
\begin{equation}\label{eq:unsere_lösungen}
\begin{aligned}
\begin{bmatrix}
\lambda\\x
\end{bmatrix}
& = \exp{-[N_\mu^D(I-\mu N_\mu)](\,\cdot\,-t_0)}
\begin{bmatrix}
\lambda(t_0)\\
x(t_0)
\end{bmatrix},\\
u & = M_\mu N_\mu^D\begin{bmatrix}
\lambda\\x
\end{bmatrix}.
\end{aligned}
\end{equation}
\end{Prop}
\begin{proof}
In view of the calculations in~\cite[pp.\,1096]{Camp76} and Proposition~\ref{prop:representation_solution},  $\begin{bmatrix}\lambda^T & x^T & u^T\end{bmatrix}^T$ is a solution of~\eqref{eq:the_optimality_DAE} if, and only if, there are $v_\lambda,v_x\in\mathbb{K}^n$ so that
\begin{align}
\begin{split}\label{eq:Campy}
\begin{bmatrix}
\lambda\\x
\end{bmatrix} & = \exp{-[N_\mu^D(I-\mu N_\mu)](\cdot-t_0)}N_\mu^DN_\mu\begin{bmatrix}
v_\lambda\\
v_x
\end{bmatrix},\\
u & = M_\mu N_\mu^D\begin{bmatrix}
\lambda\\x
\end{bmatrix}.
\end{split}
\end{align}
 Note that $N_\mu^DN_\mu$ is the spectral projection onto the sum of the root subspaces of $N_\mu$ corresponding to the non-zero eigenvalues.
Therefore, every function with $\begin{bmatrix}
   \la(t_0)\\ x(t_0)) 
\end{bmatrix}\in\im(N_\mu^DN_\mu)$
that has the representation~\eqref{eq:unsere_lösungen} is indeed a solution of~\eqref{eq:the_optimality_DAE}. Conversely, let $\begin{bmatrix}\lambda^T & x^T & u^T\end{bmatrix}^T\in\mathcal{C}^1([t_0,t_1],\mathbb{K}^{n+n+m})$ be a solution of~\eqref{eq:the_optimality_DAE} and let $v_\lambda,v_x\in\mathbb{R}$ be such that \eqref{eq:Campy} holds.
Then $\begin{bmatrix}\lambda^T(t_0) \\ x^T(t_0) \end{bmatrix} = N_\mu^DN_\mu\begin{bmatrix} v_\la\\ v_x\end{bmatrix}\in\im(N_\mu^D N_\mu)$, which proves the assertion.
\end{proof}

The statement of Proposition \ref{prop:initial_conditions} shows that the 
subspace $\im(N_\mu^DN_\mu)$ determines the set of admissible initial values for the boundary value problem \eqref{eq:the_optimality_DAE}. Since \eqref{eq:the_optimality_DAE} does not depend on the particular choice of $\mu$, the subspace $\im(N_\mu^DN_\mu)$ does not either, which is shown in the following corollary.

\begin{Cor}
Let $(J,R,Q,B)\in\PH$ and $\Q\in\mathbb{K}^{m\times m}$ with $\Q^H = \Q\geq 0$ be such that $s\mathcal{E}-\mathcal{A}_\Q$ is regular. Choose $\mu,\nu\in\mathbb{K}\setminus(\sigma((J-R)Q)\cup \sigma(Q(J+R)))$ such that both $\mu\mathcal{E}-\mathcal{A}_\Q$ and $\nu\mathcal{E}-\mathcal{A}_\Q$ are invertible. Then we have $\im(N_\mu^DN_\mu) = \im(N_\nu^D N_\nu)$.
\end{Cor}


The following proposition finally characterizes the existence of a solution of the optimality system with prescribed initial and terminal values for the state, cf.\ \cite[Theorem 4]{Camp76}.

\begin{Prop}\label{prop:existence}
Let $(J,R,Q,B)\in\PH$ and $\Q\in\mathbb{K}^{m\times m}$ with $\Q^H = \Q\geq 0$. Let $\mu\in\mathbb{K}\setminus(\sigma((J-R)Q)\cup \sigma(-(J-RQ)^H))$ be such that $\mu\mathcal{E}-\mathcal{A}_\Q$ is invertible and let $E_1,E_2,E_3,E_4:\mathbb{R}\to\mathbb{K}^{n\times n}$ be such that for all $t\in\mathbb{R}$
\begin{align}\label{eq:drazindecomp}
 \exp{-[N_\mu^D(I-\mu N_\mu)](t-t_0)}
= \begin{bmatrix}
E_1(t) & E_2(t)\\
E_3(t) & E_4(t)
\end{bmatrix}.
\end{align}
If $E_3(t_1)$ is invertible, then~\eqref{eq:the_optimality_DAE} possesses a solution $\begin{bmatrix}\lambda^T & x^T & u^T\end{bmatrix}^T\in\mathcal{C}^1([t_0,t_1],\mathbb{K}^{n+n+m})$ with $x(t_0) = x^0$ and $x(t_1) = x^1$ if, and only if,
\begin{align}\label{e:x0x1_condi}
\begin{bmatrix}E_3(t_1)^{-1}[x^1 - E_4(t_1)x^0]\\ x^0
\end{bmatrix}\,\in\,\im(N_\mu^DN_\mu).
\end{align}
\end{Prop}
\begin{proof}
By Proposition \ref{prop:initial_conditions}, the system~\eqref{eq:the_optimality_DAE} has a solution $(x,\lambda,u)$ with $x(t_0) = x^0$ if, and only if, $\begin{bmatrix}\lambda^T(t_0) \\ x^T(t_0) \end{bmatrix} = N_\mu^DN_\mu\begin{bmatrix} v_\la\\ v_x\end{bmatrix}\in\im(N_\mu^D N_\mu)$ and \eqref{eq:unsere_lösungen} holds. If we evaluate the state of such a solution at time $t_1$, we obtain
\begin{align*}
x(t_1) = E_3(t_1)\lambda(t_0) + E_4(t_1)x^0,
\end{align*}
and hence $x(t_1) = x^1$ if, and only if, $\la(t_0) = E_3(t_1)^{-1}[x^1 - E_4(t_1)x^0]$.
\end{proof}

The optimal control that we have obtained in Proposition~\ref{prop:initial_conditions}, cf.\ \eqref{eq:unsere_lösungen}, is not a state feedback solution, as it depends additionally on the adjoint. However, Proposition~\ref{prop:existence} enables us to eliminate this dependence and to define the optimal control as a state feedback.

\begin{Cor}
    Let $(J,R,Q,B)\in\PH$ and $\Q\in\mathbb{K}^{m\times m}$ with $\Q^H = \Q\geq 0$. Let $\mu\in\mathbb{K}\setminus(\sigma((J-R)Q)\cup \sigma(-(J-RQ)^H))$ be such that $\mu\mathcal{E}-\mathcal{A}_\Q$ is invertible. Consider the decomposition \eqref{eq:drazindecomp}, assume that $E_3^\mu(t_1)$ is invertible and let the condition \eqref{e:x0x1_condi} hold. 
    
    Then, the unique optimal control of \eqref{eq:ultimate_port-Hamiltonian_OCP} is given by
    \begin{align*}
        u(t) = M_\mu N_\mu^D\exp{-[N_\mu^D(I-\mu N_\mu)](t-t_0)}
\begin{bmatrix}
E_3^\mu(t_1)^{-1}[x^1 - E_4^\mu(t_1)x^0]\\
x^0
\end{bmatrix}
\end{align*}
for all $t\in [t_0,t_1]$ with $N_\mu$ and $M_\mu$ as defined in \eqref{eq:Nmu} and \eqref{eq:Mmu}.
\end{Cor}
\begin{proof}
The formula \eqref{eq:unsere_lösungen} implies that
\begin{align*}
        u = M_\mu N_\mu^D\exp{-[N_\mu^D(I-\mu N_\mu)](\cdot-t_0)}
\begin{bmatrix}
\lambda(t_0)\\
x(t_0)
\end{bmatrix}.
\end{align*}
Inserting the formula $\lambda(t_0) = E_3^\mu(t_1)^{-1}[x^1 - E_4^\mu(t_1)x^0]$ derived in the proof of Proposition~\ref{prop:existence} yields the assertion.
\end{proof}

\section{Illustrating examples}\label{sec:examples}

In this section, we provide several examples to illustrate the theoretical results of the previous sections. 
As shown in Proposition~\ref{prop:initial_conditions}, more specifically in \eqref{eq:unsere_lösungen}, the flow of the optimality DAE \eqref{eq:the_optimality_DAE} can be calculated by means of the matrix function $H_\mu : \R\to\K^{2n\times 2n}$, defined by
$$
H_\mu(t) := \exp{-[N_\mu^D - \mu N_\mu^DN_\mu](t-t_0)}N_\mu^DN_\mu.
$$

\begin{Ex}\label{ex:echtes_beispiel}
First, we revisit Example~\ref{ex:mechsys} and set $\ell = 1$ and $M = K = 1$, i.e.\ we consider the matrices
\begin{align*}
J := \begin{bmatrix}
0 & -1\\
1 & 0
\end{bmatrix},\quad R := \begin{bmatrix}
d & 0\\
0 & 0
\end{bmatrix},\quad Q := I_2,\quad B:= \begin{bmatrix}
1\\0
\end{bmatrix}
\end{align*}
for some damping coefficient $d>0$ and the corresponding minimal energy optimal control problem, cf.\ \eqref{eq:ultimate_port-Hamiltonian_OCP} with $S=0$:
\begin{align*}
    \min\ & d\int_{t_0}^{t_1} x_1^2\,\mathrm{d}t \qquad\mathrm{s.t.}\qquad \tfrac{\mathrm{d}}{\mathrm{d}t}{x} = \begin{bmatrix}
-d & -1\\
1 & 0
\end{bmatrix}x + \begin{bmatrix}
1\\0
\end{bmatrix}u,\quad x(t_0) = x^0,\ x(t_1) = x^1.
\end{align*}
The associated optimality pencil is then given by
\begin{align*}
s\mathcal{E}-\mathcal{A}_0 = \begin{bmatrix}
s-d & 1 & d & 0 & 0\\
-1 & s & 0 & 0 & 0\\
0 & 0 & s+d & 1 & -1\\
0 & 0 & -1 & s & 0\\
1 & 0 & 0 & 0 & 0
\end{bmatrix}.
\end{align*}
As already noted in Example~\ref{ex:mechsys}, the pencil is regular since the condition $\im B\cap\ker RQ = \{0\}$ is satisfied, cf.\ Proposition~\ref{prop:Schaller_Philipp_für_uns}. 
The flow of the associated DAE can be calculated along the lines of Proposition~\ref{prop:initial_conditions}, 
\begin{align*}
H_d:\mathbb{R}\to\mathbb{R}^{4\times 4},\quad t\mapsto \exp{-[N_d^D-dN_d^DN_d](t-t_0)}N_d^DN_d = \frac{1}{d}\begin{bmatrix}
0 & 0 & 0 & 0\\
0 & d & 0 & 0\\
0 & -1 & 0 & 0\\
1 & -t & 0 & d
\end{bmatrix},
\end{align*}
cf.\ Appendix~\ref{sec:example_1} for details.
This flow maps an initial state and adjoint state
\begin{align*}
\begin{bmatrix}(\lambda^0)^\top& (x^0)^\top\end{bmatrix}^\top\in\im(N_d^DN_d) = \set{\begin{bmatrix}0 & -dz & z & \zeta\end{bmatrix}^\top\,\big\vert\,z,\zeta\in\mathbb{R}}
\end{align*}
to the trajectories of the adjoint $\lambda$ and the state $x$. Especially, we conclude that each initial value~$x^0$ admits a unique $\lambda^0\in\mathbb{R}^2$ so that the optimality system corresponding to $s\mathcal{E}-\mathcal{A}_0$
possesses a solution with initial conditions $\lambda(t_0) = \lambda^0$ and $x(t_0) = x^0$. In particular, an initial condition $\begin{bmatrix}
    0 & -dx^0_1 & x^0_1 & x^0_2
\end{bmatrix}$ generates the state trajectory
\begin{align}\label{eq:formulastate}
x(t) & = \begin{bmatrix}
x_1^0\\
t x_1^0 + x_2^0
\end{bmatrix}.
\end{align}
Therefore, the optimality system yields an optimal solution if, and only if, $x^1_1 = x_1^0$ and $x_2^1 = t_1 x_1^0 + x_2^0$.
Directly applying Proposition~\ref{prop:existence}, 
we get from~\eqref{e:x0x1_condi} the necessary and sufficient condition
\begin{align*}
    \begin{bmatrix}
     \frac{1}{d}(x_2^1-x_2^0-t_1x^1_1)\\
     -\frac{1}{d}x_1^1\\
     x^0_1\\
     x^0_2
    \end{bmatrix} = \begin{bmatrix}
        \frac{1}{d}\begin{bmatrix}
            -t_1 & 1\\-1 & 0
        \end{bmatrix}\left(x^1-\begin{bmatrix}
            0 & 0\\0 & 1
        \end{bmatrix}x^0\right)\\
        x^0
    \end{bmatrix}\in\im N_d^DN_d = \set{\begin{bmatrix}0 & -dz
    &z & \zeta \end{bmatrix}^T\,\big\vert\, z,\zeta\in\mathbb{R}}
\end{align*}
for the existence of an optimal trajectory that transfers the initial value $x(t_0) = x^0$ to the terminal value $x(t_1) = x^1$. 
To calculate the associated optimal control, we refer to Proposition~\ref{prop:initial_conditions}. Recall that
\begin{align*}
M_d = -\Q_d^{-1}\begin{bmatrix}
    B^H &0
\end{bmatrix}E_d = \frac{1}{d^3}\begin{bmatrix}
2d^3+d & -2d^2-1 & -d^3 & d^2\end{bmatrix}.
\end{align*}
Therefore, we get the associated control
\begin{align}\label{eq:formulacontrol}
u(t) & = M_d N_d^D\begin{bmatrix}\lambda(t)\\ x(t)\end{bmatrix} = \frac{1}{d^4}\begin{bmatrix} d^3 &-d^4  &0 &d^4\end{bmatrix}\begin{bmatrix}\lambda(t)\\ x(t)\end{bmatrix} =(t+d) x_1^0+ x_2^0.
\end{align}
The formulas for the optimal state \eqref{eq:formulastate} and control \eqref{eq:formulacontrol}, and the consistent boundary conditions 
coincide with the findings of Example~\ref{ex:mechsys} for this particular scalar case. We stress, however, that in Example~\ref{ex:mechsys} we heavily exploited the structure of the second-order mechanical system, whereas the approach of Proposition~\ref{prop:initial_conditions} employed here only hinges on the port-Hamiltonian structure of the OCP.
\end{Ex}
Next we draw upon two academic examples: The first one illustrates that we can choose $\Q = 0$, while the second one shows how to compute the minimal pertubation. 

\begin{Ex}\label{eq:zweites_Beispiel}
Consider the matrices
\begin{align*}
J:=\begin{bmatrix}
i & 0\\
0 & 0
\end{bmatrix},\quad R := \begin{bmatrix}
0 & 0\\
0 & 1
\end{bmatrix},\quad Q := \begin{bmatrix}
1 & i\\
-i & 1
\end{bmatrix},\quad B := \begin{bmatrix}
1\\0
\end{bmatrix}.
\end{align*}
It is clear that $J$ is skew-Hermitian, that $R$ and $Q$ are Hermitian and positive semidefinite.
Using Proposition~\ref{prop:final_equivalence}, we verify that $s\mathcal{E}-\mathcal{A}_0$ is regular. Let $\omega\in\mathbb{R}\setminus\set{0,1}$. Then we have
\begin{align*}
(JQ-i\omega I)^{-1}\ker RQ & = \begin{bmatrix}
i(1-\omega) & -1\\
0 & -i\omega
\end{bmatrix}^{-1}\set{\begin{bmatrix}z\\ iz\end{bmatrix}\,\big\vert\,z\in\mathbb{C}} = \set{\begin{bmatrix}-iz \\ -\frac{z}{\omega}\end{bmatrix}\,\vert\,z\in\mathbb{C}}, 
\end{align*}
see Appendix~\ref{sec:example_2} for the calculation of $\ker RQ$.
Since $B$ has full column rank and $\im B = \mathbb{C}\times\set{0}$, we have $\im B\cap (JQ-i\omega I)^{-1}\ker RQ = \set{0}$ and hence Proposition~\ref{prop:final_equivalence} yields that $s\mathcal{E}-\mathcal{A}_\Q$ is regular. 
In view of Proposition~\ref{prop:initial_conditions}, for given $t_0\in\mathbb{R}$ the solutions $(x,\lambda,u)$ have the form
\begin{equation}\label{eq:equation_3}
\begin{aligned}
\begin{bmatrix}
\lambda(t)\\x(t)
\end{bmatrix} & = T\begin{bmatrix}
\exp{\frac{1}{2}(i-\sqrt{3+8i})(t-t_0)} & 0 & 0 & 0\\
0 & \exp{\frac{1}{2}(i+\sqrt{3+8i})(t-t_0)} & 0 & 0\\
0 & 0 & \exp{(1-i)(t-t_0)} & 0\\
0 & 0 & 0 & 1
\end{bmatrix}T^{-1}\begin{bmatrix}
\lambda(t_0)\\
x(t_0)
\end{bmatrix},\\
u(t) & = -\frac{5}{1-2i} \big(x_1(t)+(2-i)x_2(t)-i\lambda_2(t)\big)
\end{aligned}
\end{equation}
for all $t\in\mathbb{R}$ with $T$ given in Appendix~\ref{sec:example_2}. Given $x^0,x^1\in\mathbb{C}^2$ and $t_0<t_1\in\mathbb{R}$, we conclude that~\eqref{eq:equation_3} has a solution $(x,\lambda,u)$ with $x(t_0) = x^0$ and $x(t_1) = x^1$ if, and only if,
\begin{align}\label{eq:equation_4}
x^1 = \frac{1}{(28i-12)}\begin{bmatrix}
(8+10i)\exp{(1-i)(t_1-t_0)}-20+8i & (8+10i)\exp{(1-i)(t_1-t_0)}+20-8i\\
(8+10i)\exp{(1-i)(t_1-t_0)}-8-20i & (8+10i)\exp{(1-i)(t_1-t_0)}+8+20i
\end{bmatrix}x^0.
\end{align}
Therefore, we can use the presented method to find an optimal solution of the associated optimal control problem~\eqref{eq:ultimate_port-Hamiltonian_OCP}
if, and only if, the boundary conditions $x^0$ and $x^1$ fulfill~\eqref{eq:equation_4}. To obtain the same result from~\eqref{e:x0x1_condi} in Proposition~\ref{prop:existence}, we have
\begin{align*}
 H_1(t_1) = \exp{A t_1} = T\exp{T^{-1}AT}T^{-1} = T\begin{bmatrix}
\exp{\frac{1}{2}(i-\sqrt{3+8i})(t_1-t_0)} & 0 & 0 & 0\\
0 & \exp{\frac{1}{2}(i+\sqrt{3+8i})(t_1-t_0)} & 0 & 0\\
0 & 0 & \exp{(1-i)(t_1-t_0)} & 0\\
0 & 0 & 0 & 1
\end{bmatrix}T^{-1},
\end{align*}
where we omit the details of the calculation.

\end{Ex}

Next, we give an example in which the pencil $s\mathcal{E}-\mathcal{A}_0$ is singular, and use Lemma~\ref{lem:rank_minimal_realization} to find a rank-minimal perturbation $\Q$ so that the perturbed pencil $s\mathcal{E}-\mathcal{A}_\Q$ is regular.

\begin{Ex}\label{ex:drittes_beispiel}
Let $J = i$, $R = 0$, $Q = 1$ and $B = \begin{bmatrix} 1&i\end{bmatrix}\in\mathbb{C}^{1\times 2}$. Then $(J,R,Q,B)\in\PH$. In this case 
\begin{align*}
s\mathcal{E}-\mathcal{A}_0 = \begin{bmatrix}
s-i & 0 & 0 & 0\\
0 & s-i & -1 & -i\\
1 & 0 & 0 & 0\\
i & 0 & 0 & 0
\end{bmatrix}.
\end{align*}
is evidently singular. Since $\ker RQ = \mathbb{C}$, we have $B^{-1}(JQ-i\omega I)\ker RQ = \mathbb{C}^2$ and by Lemma~\ref{lem:rank_minimal_realization} the pencil $\mathcal{E}-\mathcal{A}_\Q$ with $\Q\in\mathbb{K}^{m\times m}$ symmetric and positive semidefinite is regular if, and only if, $\Q$ is positive definite. 
Recall that the DAE associated with the pencil $s\mathcal{E}-\mathcal{A}_0$ aims to determine the optimal solutions of the optimal control problem
\begin{equation*}
\begin{aligned}
\min\quad & 0\\
\mathrm{s.t.}\quad & \tfrac{\mathrm{d}}{\mathrm{d}t}x = ix+u_1+iu_2,\qquad x(t_0) = x^0,\ x(t_1) = x^1
\end{aligned}
\end{equation*}
and we see that every solution of the ODE which fulfills the boundary conditions is optimal. If $t_1-t_0$ is no multiple of $2\pi$, then the function
\begin{align*}
x:\mathbb{R}\to\mathbb{C},\quad t\mapsto \exp{i(t-t_0)}x^0 + \frac{\exp{i(t_1-t_0)}x^0-x^1}{i\exp{i(t-t_0)}-i}
\end{align*}
is an optimal solution with constant control. However, our perturbed problem associated with the regular pencil $s\mathcal{E}-\mathcal{A}_\Q$ is the optimal control problem
\begin{align*}
\min\quad & \int_{t_0}^{t_1} u^H\Q u\,\mathrm{d}t\\
\mathrm{s.t.}\quad & \tfrac{\mathrm{d}}{\mathrm{d}t}x = (J-R)Qx+Bu,\qquad x(t_0) = x^0,\ x(t_1) = x^1.
\end{align*}
Consider e.g. $\Q = I_2$. Then for all $\mu\in\mathbb{C}$ we have that
$\mu\mathcal{E}-\mathcal{A}_\Q $
is invertible if, and only if, $\mu\neq i$.
Taking e.g.  $\mu = 2i$, we get
\begin{align*}
(2i\mathcal{E}-\mathcal{A}_I)^{-1} = \begin{bmatrix}
i & 0 & 0 & 0\\
0 & i & -1 & -i\\
1 & 0 & -1 & -1\\
i & 0 & 0 & -1
\end{bmatrix}^{-1} = \begin{bmatrix}
-i & 0 & 0 & 0\\
0 & -i & i & -1\\
-i & 0 & -1 & 0\\
1 & 0 & 0 & -1
\end{bmatrix},
\end{align*}
and hence,
\begin{align*}
N_{2i} = \begin{bmatrix}
-i & 0\\
0 & -i
\end{bmatrix},\quad N_{2i}^D = N_{2i}^{-1} = \begin{bmatrix}
i & 0\\
0 & i
\end{bmatrix},\quad\text{and}\quad M_{2i} = \begin{bmatrix}
-i & 0\\
1 & 0
\end{bmatrix}.
\end{align*}
Then $N_{2i}^D(I-2iN_{2i}) = N_{2i}^D-2i$ and we can calculate, for all $t\in\mathbb{R}$,
\begin{align*}
H_{2i}(t) = \exp{-[N_{2i}^D(I-2iN_{2i})](t-t_0)} = \begin{bmatrix}
\exp{i(t-t_0)} & 0\\
0 & \exp{i(t-t_0)}
\end{bmatrix}.
\end{align*}
By Proposition~\ref{prop:initial_conditions}, all solutions of the DAE $(\tfrac{\mathrm{d}}{\mathrm{d}t}\mathcal{E}-\mathcal{A}_I)(\lambda,x,u)$ with the initial condition $x(0) = x^0$ are given as
\begin{align*}
\lambda(\cdot) & = \exp{i(\cdot-t_0)}\lambda(0),\\
x(\cdot) & = \exp{i(\cdot-t_0)}x^0,\\
u(\cdot) & = \lambda(\cdot)\begin{bmatrix}
1\\i
\end{bmatrix}.
\end{align*}
In particular, we have a solution of the optimality DAE that fulfills the boundary values for $x$ if, and only if, $x^1 = \exp{i(t_1-t_0)}x^0$. Note, that Proposition~\ref{prop:existence} is not applicable as $E_3(t)=0$ is not invertible in this case.
\end{Ex}
As last example, we  consider 
a boundary controlled heat equation.
\begin{Ex}\label{ex:heat}
We consider the temperature distribution in a domain $\Omega\subset \mathbb{R}^d$, $d\in \{1,2,3\}$ with a Dirichlet boundary control on a part of the boundary $\Gamma_\mathrm{c}\subset \partial \Omega$. The model is given by
\begin{align*}
   && \dot \theta(t,\omega)&= \kappa \Delta \theta(t,\omega),\ &&\omega\in \Omega, \ t\in [0,T],\\
   && \theta(t,\gamma) &= u(t,\gamma),\ &&\gamma \in \Gamma_\mathrm{c},\ t\in [0,T],\\ 
   && \theta(s,\gamma) &= 0,\ &&\gamma \in \Gamma\setminus \Gamma_\mathrm{c}, t\in [0,T],
\end{align*}
where $\theta (t,\omega)$ is the temperature at time $t\in [0,T]$ at the space point $\omega\in \Omega$ and $\kappa > 0$ is a heat conductivity coefficient.
Discretizing in space by central finite differences, see e.g.
\cite{Str04}, with a regular mesh of mesh size $h> 0$ and $n\in \mathbb{N}$ grid points $\{\omega_0,\ldots,\omega_{n-1}\}$, leads to a control problem
\begin{equation*}
    \dot z_h(t)= -R_h z_h(t) + B_h u(t),\  t\in [0,T],
\end{equation*}
where $z_h(t)\in \R^{n}$ consists of the nodal values of the temperature profile, i.e. $(z_h(t))_j = \theta(t,\omega_j)$ for $j\in \{0,\ldots,n-1\}$.  The matrix $R_h\in \R^{n \times n}$ is the finite difference
matrix associated with the negative Laplace operator at all mesh points. The input matrix $B_h \in \R^{n\times m}$, describes the matrix with associated with the $m\in \mathbb{N}$ grid points on the control boundary $\Gamma_\mathrm{c}$. Adding an output equation $y=B_h^T z_h$ yields a port-Hamiltonian system of the form \eqref{eq:pH_ISO-system} with $R=R_h$, $J=0$, $Q=I$ and $B=B_h$,
%
%
where $R = R_h$ 
is symmetric and positive semidefinite with a kernel of the dimension of the number of grid points on the boundary. However, as the values on the boundary points are fixed, the corresponding rows of $B_h$ are zero, such that $\im B_h\cap \ker R_h = \{0\}$.

We briefly illustrate this for the one-dimensional case with $\Omega = [0,1]$, $\Gamma_\mathrm{c} = \{0\}$ and $\Gamma\setminus \Gamma_\mathrm{c} = \{1\}$  with equidistant grid $0 = \omega_0<\ldots<\omega_{n-1} = 1$.
For $t \in [0,T]$ we obtain the finite-difference approximation at the interior nodes
\begin{align*}
    \Delta \theta(t,\omega_j) \approx \frac{\theta(t,\omega_{j+1}) - 2\theta(t,\omega_{j}) + \theta(t,\omega_{j-1})}{h^2}, \qquad  j = 1,\ldots,n-2.
\end{align*}
Moreover, incorporating the boundary conditions $\theta(t,\omega_0) = u(t)$ and $\theta(t,\omega_{n-1}) = 0$, we get
\begin{align*}
        \Delta \theta(t,\omega_1) \approx \frac{\theta(t,\omega_{2}) - 2\theta(t,\omega_{1}) + u(t)}{h^2} \quad 
        \mathrm{and} \quad \Delta \theta(t,\omega_{n-2}) \approx - \frac{2\theta(t,\omega_{n-2}) - \theta(t,\omega_{n-3})}{h^2} \\
\end{align*}
and thus
\begin{align*}
R_h =
\begin{bmatrix}
0 & 0 & 0 & 0 & \cdots & 0\\
0 & 2 & -1 & \ddots & \ddots & \vdots\\
0 & -1 & \ddots & \ddots & \ddots & 0\\
0 & \ddots & \ddots & \ddots & -1 & 0\\
\vdots & \ddots & \ddots & -1 & 2 & 0\\
0 & \cdots & 0 & \phantom{\ddots}\hspace*{-.4cm}0 & 0 & 0
\end{bmatrix}\in \R^{n\times n}, \qquad 
B_h = \begin{bmatrix}
    0\\1\\0\\\vdots\\0
\end{bmatrix}\in \R^{n\times 1}.
\end{align*}
Here, clearly $\ker R_h = \operatorname{span}\{e_1,e_{n}\}$ and $\im B_h = \operatorname{span}\{e_2\}$, where $e_j\in \R^{n}$, $j\in \{1,\ldots,n\}$ 
is the canonical unit vector. Thus, $\im B_h \cap \ker R_h = 0$. Considering the corresponding optimal control problem with minimal energy supply, applying Proposition~\ref{prop:Schaller_Philipp_für_uns} yields regularity of the corresponding optimality system and Proposition~\ref{prop:indexthree} implies that the associated DAE has index three. In particular this shows that a regularization term in the cost penalizing the control is not necessary.
\end{Ex}

\section{Conclusions and outlook}\label{sec:conclusions}

The optimality system associated with the problem of steering a port-Hamiltonian system from an initial state to a terminal state on a given time interval with minimal energy supply has been studied. Regularity of this differential-algebraic equation has been characterized exploiting the port-Hamiltonian structure. 
Minimal rank quadratic perturbations of the cost functional have been characterized so that the perturbed problem produces a regular optimality system. Assuming regularity, the solution formula for regular linear time-invariant differential-algebraic equations leads to a characterization of optimal solutions. Moreover, we have derived conditions which allow to characterize the optimal control as a state feedback. Last, the results are illustrated by means of examples, including second-order mechanical systems and a discretized heat equation.

Future works consider extensions to control of nonlinear or infinite-dimensional port-Hamiltonian systems, such as beam equations, building upon \cite{lhmnc} are subject to future research. In the nonlinear case, first steps were conducted by \cite{Kars23,Masc22}. In the infinite-dimensional case, we expect that the structure of the descriptor matrix in the optimality DAE can help to generalize the results. 

\appendix

\section{Detailed Calculations}

\subsection{Calculations for Example~\ref{ex:echtes_beispiel}}\label{sec:example_1}

We calculate $\Q_\mu$ at $\mu = d$. In view of~\cite[Formula (35)]{Camp76}, we have
\begin{align*}
S_d & = -B^\top(d-J-R)^{-1}R(d-J+R)^{-1}B\\
& = -[1,0]\begin{bmatrix}
0 & 1\\-1 & d
\end{bmatrix}^{-1}\begin{bmatrix}
d & 0\\
0 & 0
\end{bmatrix}\begin{bmatrix}
2d & 1\\-1 & d
\end{bmatrix}^{-1}\begin{bmatrix}
1\\0
\end{bmatrix}  = -\frac{d^3}{2d^2+1},
\end{align*}
and after some tedious, but straightforward calculations
\begin{align*}
E_d = \frac{1}{2d^2+1}\begin{bmatrix}
2d^3+d & -2d^2-1 & -d^3 & d^2\\
2d^2+1 & 0 & -d^2 & d\\
0 & 0 & d & -1\\
0 & 0 & 1 & 2d
\end{bmatrix}, \qquad N_d = \frac{1}{d^3}\begin{bmatrix}
0 & 0 & 0 & 0\\
0 & d^2 & 0 & 0\\
d^2 & -d & 0 & 0\\
d & -1 & 0 & d^2
\end{bmatrix}
\end{align*}
using the definition $N_d = 
E_d+E_d\begin{bmatrix}B\\0\end{bmatrix}S_d^{-1}\begin{bmatrix} B^H & 0\end{bmatrix}E_d$.
To calculate the Drazin inverse, we transform~$N_d$ in Jordan form
\begin{align*}
    S^{-1}N_dS = \begin{bmatrix}
        0 & 1 & 0 & 0\\
        0 & 0 & 0 & 0\\
        0 & 0 & \frac{1}{d} & 1\\
        0 & 0 & 0 & \frac{1}{d}
\end{bmatrix} \qquad\text{with}\qquad 
S := \begin{bmatrix}
0 & d & 0 & 0\\
0 & 0 & 0 & -d^3\\
1 & 0 & 0 & d^2\\
0 & -1 & 1 & 0
\end{bmatrix}.
\end{align*}
Hence, its Drazin inverse can be calculated as
\begin{align*}
N_d^D & = S\begin{bmatrix}
0 & 0 & 0 & 0\\
0 & 0 & 0 & 0\\
0 & 0 & d & -d^2\\
0 & 0 & 0 & d
\end{bmatrix}S^{-1} = \frac{1}{d}\begin{bmatrix}
0 & 0 & 0 & 0\\
0 & d^2 & 0 & 0\\
0 & -d & 0 & 0\\
d & 1 & 0 & d^2
\end{bmatrix}.
\end{align*}
Thus, we get
\begin{align*}
N_d^D-dN_d^DN_d = \frac{1}{d}\begin{bmatrix}
0 & 0 & 0 & 0\\
0 & 0 & 0 & 0\\
0 & 0 & 0 & 0\\
0 & 1 & 0 & 0
\end{bmatrix}.
\end{align*}

\subsection{Calculations for Example~\ref{eq:zweites_Beispiel}}\label{sec:example_2}

To verify that $Q = \begin{bmatrix}1 & i\\-i & 1\end{bmatrix}$ is positive semidefinite, we calculate the characteristic polynomial as $\chi_Q(s) = \det(s-Q) = (s-1)^2-1 = s(s-2)$ and hence $\sigma(Q) = \set{0,2}$ so that $Q$ is indeed positive semidefinite. The kernel of $RQ$ can be easily calculated directly from
\begin{align*}
    \ker RQ = \ker\begin{bmatrix}
        0 & 0\\
        -i & 1
    \end{bmatrix} = \set{(z,iz)\,\big\vert\, z\in\mathbb{C}}.
\end{align*}
To solve~\eqref{eq:the_optimality_DAE}, we calculate $\Q_\mu$ for $\mu = 1$. By~\cite[Formula~(35)]{Camp76}, we have
\begin{align*}
\Q_1 & = -B^H(I-Q(J+R))^{-1}Q(I-(J-R)Q)^{-1}B = \frac{1-2i}{5}\neq 0.
\end{align*}
Further, we have
\begin{align*}
E_1 = \begin{bmatrix}
0 & -1 & \frac{1-2i}{5} & \frac{2+i}{5}\\
i & 1+i & -\frac{2+i}{5} & \frac{1-2i}{5}\\
0 & 0 & \frac{4+2i}{5} & -\frac{2+i}{5}\\
0 & 0 & -\frac{1-2i}{5} & \frac{3-i}{5}
\end{bmatrix} \qquad\text{and}\qquad
N_1^D = N_1^{-1} = \begin{bmatrix}
2+i & -i & 1 & i\\
-1 & 0 & -i & 1\\
0 & 0 & 1-i & 1\\
0 & 0 & -i & 2
\end{bmatrix}.
\end{align*}
Hence, we have $A := N_1^D(I_4- N_1) = N_1^D-I_4$. 
The characteristic polynomial of $A$ is given by
\begin{align*}
\chi_A(s) = s(s+i-1)\left(s-\frac{1}{2}(i+\sqrt{3+8i})\right)\left(s-\frac{1}{2}(i-\sqrt{3+8i})\right).
\end{align*}
Thus, $A$ is diagonalizable and a direct calculation verifies that the transformation matrix
\begin{align*}
T & :=\frac{1}{\sqrt{3+8i}(28i-12)}\begin{bmatrix}
2 & 2 & 2 & 0\\
1-i(2+\sqrt{3+8i}) & 1-i(2-\sqrt{3+8i}) & -3-3i & 0\\
0 & 0 & -2-5i & 1\\
0 & 0 & -2-5i & i
\end{bmatrix}
\end{align*}
diagonalizes the matrix $A$ as $T^{-1}AT = \mathrm{diag}(\frac{1}{2}(i-\sqrt{3+8i}), \frac{1}{2}(i+\sqrt{3+8i}), 1-i, 0)$.


\section*{Acknowledgement}
We are indebted to our colleagues Achim Ilchmann and Timo Reis for several constructive discussions.

\section*{Funding}
JK thanks Technische Universität Ilmenau and Freistaat Thüringen for their financial support as part of the Thüringer Graduiertenförderung. VM was supported by German Research Foundation through Priority Program SPP 1984 \textit{Hybrid and multimodal energy systems} within the project \textit{Distributed Dynamic Security Control}. FP was funded by the Carl Zeiss Foundation within the project \textit{DeepTurb--Deep Learning in and from Turbulence} and by the free state of Thuringia within the project \textit{THInKI--Th\"uringer Hochschulinitiative für KI im Studium}. KW gratefully acknowledges funding by the German Research Foundation (DFG; grant no.\ 507037103).

\bibliographystyle{siamplain}
\bibliography{Arxiv}
\end{document}